\documentclass[12pt,reqno]{amsart}
\usepackage{a4wide,amsfonts,amsmath,latexsym,amssymb,euscript,eufrak,graphicx,units,mathrsfs}
\usepackage[utf8]{inputenc}
\usepackage{amsmath}
\usepackage{amsfonts}
\usepackage{amssymb}
\usepackage{amsthm}
\usepackage{floatrow}
\usepackage{blindtext}
\usepackage{multicol}
\usepackage[english]{babel}
\usepackage{enumerate}
\usepackage{eufrak}
\usepackage{graphicx}
\usepackage{caption}
\usepackage{subcaption}
\usepackage{float}
\usepackage{epstopdf}
\usepackage{multirow}
\usepackage{mathtools}
\usepackage{multirow}
\usepackage[usenames,dvipsnames]{color}

\numberwithin{equation}{section}

\newcommand{\R}{\mathbb{R}}

\newtheorem{theorem}{Theorem}[section]
\newtheorem{lemma}[theorem]{Lemma}

\theoremstyle{definition}

\theoremstyle{remark}
\newtheorem{remark}[theorem]{Remark}

\numberwithin{equation}{section} 

\begin{document}
\title[Gaussian and Hermite Ornstein-Uhlenbeck processes]{Gaussian and Hermite Ornstein-Uhlenbeck processes}
\thanks{
Khalifa Es-Sebaiy\\
khalifa.essebaiy@ku.edu.kw\\ \\ Kuwait University, Faculty of
Science, Department of Mathematics, Kuwait City, Kuwait\\ \\ ORCID
ID: 0000-0003-4915-6850\\ \\
\\This project was funded by Kuwait
Foundation for the Advancement of Sciences (KFAS) under project
code: PR18-16SM-04.}
\author[K. Es-Sebaiy]{ Khalifa Es-Sebaiy}
\address{Department of Mathematics, Faculty of Science, Kuwait University, Kuwait}
\email{khalifa.essebaiy@ku.edu.kw}

\begin{abstract}
In the present paper we  study the asymptotic behavior  of the
auto-covariance function for Ornstein-Uhlenbeck (OU) processes
driven by Gaussian noises with stationary and non-stationary
increments and for Hermite OU processes. Our results are
generalizations of the corresponding results of  Cheridito et al.
\cite{CKM} and Kaarakka and Salminen \cite{KS}.
\end{abstract}

\maketitle

\medskip\noindent
{\bf Mathematics Subject Classifications (2020)}: Primary 60H10;
Secondary: 60G15, 60G10, 60G22.

\medskip\noindent
{\bf Keywords:} Gaussian and Hermite Ornstein-Uhlenbeck processes;
auto-covariance function; stationarity and ergodicity.

\allowdisplaybreaks

\renewcommand{\thefootnote}{\arabic{footnote}}

\section{Introduction}

Let  $B^H:= \left\{B^H_t, t\in\R\right\}$ denote a fractional
Brownian motion (fBm) with Hurst parameter $H\in(0,1)$, that is,
$B^H$ is a centered Gaussian process with covariance function
\[
E\left(B_{t}^{H} B_{s}^{H}\right)=\frac{1}{2}\left(|t|^{2 H}+|s|^{2
H}-|t-s|^{2 H}\right),\quad t, s \in \mathbb{R}. \] Consider the
fractional Ornstein-Uhlenbeck (fOU) process $X^H:= \left\{X^H_t,
t\in\R\right\}$ defined  as the solution to the Langevin equation
\begin{eqnarray}
  dX_{t}^H = -\theta X_t^H dt+dB_{t}^H,
  \label{FOU-model}
\end{eqnarray}
where $\theta\in \R$.\\
Notice that if $\theta>0$ and the initial condition
$X_0^H=\int_{-\infty}^0e^{\theta s}dB_{s}^H$, then the unique strong
solution $X^H$ of \eqref{FOU-model} is a stationary Gaussian
sequence and since $E(X_t^HX_0^H)\rightarrow0$ as
$t\rightarrow\infty$ (according to \cite[Theorem 2.3]{CKM}), a
criterion for stationary Gaussian processes (see \cite[Theorem
6.6]{lindgren}) gives that the process is ergodic.

The fOU process is one of the most studied and widely  applied
 stochastic process. It represents interesting model for stochastic dynamics with memory,
  with applications to e.g.  finance, telecommunication networks and physics.
In the finance context,  several researchers  in recent years have
been interested in studying statistical estimation problems for fOU
processes. The statistical analysis of equations driven by fBm
 is obviously more recent. The development of stochastic
calculus with respect to the fBm allowed to study such models. On
the other hand, the long-range dependence property makes the fBm
important driving noise in modeling several phenomena arising, for
instance, from volatility modeling in finance.
 Let us mention some  important results in this field
where  the volatility exhibits long-memory, which means that the
volatility today is correlated to past volatility values with a
dependence that decays very slowly. The authors of
\cite{CV12a,CV12b,CCR,CR98} considered the problem of option pricing
under a stochastic volatility model that exhibits long-range
dependence. More precisely they assumed that the dynamics of the
volatility are described by the equation \eqref{FOU-model}, where
the  Hurst parameter $H$ is greater than $1/2$. On the other hand,
the paper \cite{GJR} on rough volatility contends that the
short-time behavior indicates that the Hurst parameter $H$ in the
volatility is less than $1/2$. Furthermore, the drift parameter
estimation for  fractional-noise-driven Ornstein-Uhlenbeck
processes, in particular fOU, has also attracted the interest of
many researchers recently. We refer the interested readers to
\cite{HN,EV,DEV,SV,BET,KM,EE,AAE,DEKN} and references therein.

 There are two possible definitions for  the  fOU process. One can define
  this stochastic process as the solution of \eqref{FOU-model}.
Alternatively, one can define the fOU process as the Lamperti
transform of the fBm $B^H$:
\begin{equation}
\label{lamperti} Z^{\theta,H}_{t} = e ^{-\theta t} B^H _{
\frac{H}{\theta} e ^{\frac{\theta}{H} t}}, \quad t\geq 0,
\end{equation}
which is a stationary process. In the case when $H=\frac{1}{2}$,
$\theta>0$, the process (\ref{FOU-model}), with
$X_{0}=\int_{-\infty}^0e^{\theta t}dB^{\frac12}_t$, and the process
(\ref{lamperti})   have the same finite dimensional distributions,
thanks  to the L\'evy
characterization theorem.\\
 On the other hand, when
$H\not= \frac{1}{2}$, the probability distributions of
(\ref{FOU-model}) and (\ref{lamperti}) are different.  The process
given by (\ref{lamperti}) can be also be expressed as the solution
to the Langevin-type stochastic equation,
\begin{equation}
\label{FOUsk-theta}
 dX_t= -\theta X_t dt+dY_{t,B^H}^{(\theta)},\quad t\geq0,
\end{equation}
with initial condition $ X_0= B^H_{a^{(\theta)}_0}$, where the noise
$ \{Y_{t,B^H}^{(\theta)},t\geq0\}$ is given by the formula
$Y_{t,B^H}^{(\theta)}:=\int_0^t e^{-\theta s}dB^H_{a^{(\theta)}_s}$
with $a^{(\theta)}_t:=\frac{H}{\theta}e^{\frac{\theta}{H}t}$.
Motivated by this, \cite{KS} introduced a {\it fractional
Ornstein-Uhlenbeck process of the second kind} as the solution to a
 Langevin-type stochastic equation, namely
\begin{equation}
\label{FOUsk}
 dX_t= -\theta X_t dt+dY_{t,B^H}^{(1)},\quad t\geq0,
\end{equation}
with initial condition $ X _0= 0$, where the noise
$Y_{t,B^H}^{(1)}=\int_0^t e^{-s}dB^H_{a_s}$ with
$a_t:=a^{(1)}_t=He^{\frac{t}{H}}$. Whereas, the process given by
(\ref{FOU-model})  is usually called the {\it fractional
Ornstein-Uhlenbeck process of the first kind}.

Let us describe what is proved in \cite{CKM} and \cite{KS}, about
the processes \eqref{FOU-model} and \eqref{FOUsk}, respectively. In
\cite{CKM}, Cheridito et al.  showed  that the solution to Langevin
equation \eqref{FOU-model}, with $X_{0}=\int_{-\infty}^0e^{\theta
t}dB^{H}_t$, is stationary  and the decay of its auto-covariance
function behaves as a power function. However, in \cite{KS},
Kaarakka and Salminen  proved when $H>\frac12$ that the solution of
the Langevin equation \eqref{FOUsk}, with
$X_{0}=\int_{-\infty}^0e^{\theta t}dY_{t,B^H}^{(1)}$ is a stationary
process and its auto-covariance function  decays
exponentially.\\
The purpose of this paper is to provide a general approach to study
these properties for more general Gaussian and Hermite processes
that are of the form
\begin{equation*}
  dX_t=-\theta X_tdt+dG_t. \label{GH-OU}
\end{equation*}
We will study  stationarity and ergodicity properties, and   the
decay of the auto-covariance function of such processes. These facts
play an important role in stochastic analysis and in different
applications, and for these reasons the topic has been extensively
studied in the literature. For instance, they can be used to study
different parameters describing such Gaussian or Hermite processes.
Of particular interest for turbulence theory is the large and small
lags limit behavior of the auto-covariance function of the fOU
process, which has been proposed as a representation of homogeneous
Eulerian turbulent velocity, see \cite{shao}.

 The paper is organized as follows. In Section 2 we   provide some
keys lemmas needed in order to state the main results of the present
paper. In Section 3 we study the ergodicity and stationarity
properties and the decay of the auto-covariance function for Hermite
Ornstein-Uhlenbeck processes of the first kind and for Gaussian
Ornstein-Uhlenbeck processes of the second kind. In Section 4  we
end the paper by studying the auto-covariance function
 for Ornstein-Uhlenbeck processes driven by a Gaussian noise with non-stationary increments.\\

Throughout the paper,   the symbol $C$ stands for a generic
constant, whose value can change from one line to another. Moreover,
for $t \rightarrow  \infty$, we will write $f(t) \sim g(t)$,
provided that $f(t) / g(t) \rightarrow 1$.

\section{ Key lemmas}
Let $G:=\{G_t,t\in\R\}$ be a  measurable process defined on some
probability space $(\Omega, \mathcal{F}, P)$ (here, and throughout
the text, we assume that $\mathcal{F}$ is the sigma-field generated
by $G)$. The following assumption is required.
\begin{itemize}
\item[$(\mathcal{A})$]  $G_0=0$ a.s., and there exists a
constant $\gamma\in(0,1)$ such that  for every $q\geq2$ there is a
constant $c_q>0$ satisfying
\begin{eqnarray}E\left[\left|G_{t}-G_{s}\right|^{q}\right] \leq c_q|t-s|^{q
\gamma} \quad \mbox{ for all } s, t \in\R.\label{hypercontractivity}
\end{eqnarray}
\end{itemize} Note that, if
$\left(\mathcal{A}\right)$ holds, then by the Kolmogorov-Centsov
theorem, we can conclude that for all $\varepsilon \in(0, \gamma)$,
the process $G$ admits a modification with
$(\gamma-\varepsilon)-$H\"older continuous paths, still denoted $G$
in the sequel.
\begin{remark}\label{wiener-chaos-hypercontractivity} Corollary 2.8.14 in \cite{NP-book} tells us that, inside a fixed Wiener chaos, all the
$L^{q}$-norms are equivalent. Hence, if $G$ is a Gaussian or Hermite
process, the assumption $\left(\mathcal{A}\right)$ is equivalent to
$G_0=0$ a.s., and
\begin{eqnarray*}E\left[\left|G_{t}-G_{s}\right|^{2}\right] \leq
C|t-s|^{2 \gamma} \quad \mbox{ for all } s, t \in\R.
\end{eqnarray*}
\end{remark}

Let us recall some basic elements of pathwise Riemann-Stieltjes
integral, which are helpful for some of the arguments we use. For
any $\alpha\in (0,1]$, a function $f:[a,b]\to\R$ is said to be
$\alpha$-H\"older continuous function if
\[\sup_{a\leq s<t\leq b}\frac{|f(t)-f(s)|}{|t-s|^{\alpha}}<\infty.
\]
If $f, g:[a, b] \longrightarrow \mathbb{R}$ are H\"older continuous
functions of orders $\alpha$ and $\beta$ respectively with
$\alpha+\beta>1$, Young \cite{Young} proved that the
Riemann-Stieltjes integral $\int_{a}^{b} f_{s} d g_{s}$ exists.
Hence, if $G$ is a process satisfying $(\mathcal{A})$, then   the
stochastic integral $\int_{a}^{b} u_{s} d G_{s}$ is well defined as
a pathwise Riemann-Stieltjes integral provided that the trajectories
of the process $\left\{u_{t}, t \in\R\right\}$ are $\alpha$-H\"older
continuous functions on any finite interval for some
$\alpha>1-\gamma$.
\begin{lemma}\label{lemma-Riemann integral} Let $\{G_t,t\in\R\}$ be a measurable process satisfying
the assumption $\left(\mathcal{A}\right)$ and let $\xi\in
L^0(\Omega)$. Then, for every $\theta\in\R$ and $-\infty< s<t
<\infty$,
\begin{eqnarray}
 \int_s^te^{\theta r}dG_r=e^{\theta t}G_t-e^{\theta s}G_s-\theta \int_s^te^{\theta r}G_rdr.\label{IBP-G}
\end{eqnarray}
In addition, if we assume $\theta>0$, we have  for every $t\in\R$,
$\int_{-\infty}^te^{\theta r}dG_r$ is well defined as a
Riemann-Stieltjes integral and
\begin{eqnarray}
 \int_{-\infty}^te^{\theta r}dG_r=e^{\theta t}G_t -\theta \int_{-\infty}^te^{\theta r}G_rdr,\label{IBP-G-infty}
\end{eqnarray}
and the unique continuous solution to the equation
\begin{eqnarray}
X_t=\xi-\theta \int_{0}^{t} X_s d s+  G_{t}, \quad t \geq
0,\label{linear-GOU}
\end{eqnarray}
is given by
\begin{eqnarray}
X_t=e^{-\theta t}\left(\xi+ \int_{0}^{t} e^{\theta s} d G_{s}
\right), \quad t \geq 0.\label{solution-linear-GOU}
\end{eqnarray}
\end{lemma}
\begin{proof}Since $G$ satisfies  $\left(\mathcal{A}\right)$, the integral $\int_s^te^{\theta r}dG_r$ as stated above
is well defined as a pathwise Riemann-Stieltjes integral. So, the
claim \eqref{IBP-G} can be immediately obtained by integrating by
parts (see, e.g., \cite[Theorem 2.21]{WZ}).\\
Now, let us prove \eqref{IBP-G-infty}. Suppose that $\theta>0$.
Using the same argument as in the proof of (2.7) in \cite{EE}, we
have, for any $\gamma<\delta$, $\lim_{|t| \rightarrow
\infty}\frac{G_t}{|t|^{\delta}}=0$  almost surely. This implies that
for all $t\in\R$, $\int_{-\infty}^te^{\theta r}G_rdr$ exists as a
Riemann integral, which, by \cite[Theorem 2.21]{WZ}, implies that
the Riemann-Stieltjes integral $\int_{-\infty}^te^{\theta r}dG_r$
exists
too and \eqref{IBP-G-infty} holds.\\
Finally, a continuous function $X$ is  solution to the equation
\eqref{linear-GOU} if and only if the function $ u(t)=\int_{0}^{t}
X_s d s, \ t \geq 0,$  is solution to the linear differential
equation
$$
u^{\prime}(t)=-\theta u(t)+\xi+ G_{t}, \quad u(0)=0,
$$
which has the unique solution
$$
u(t)=e^{-\theta t} \int_{0}^{t} e^{\theta s}\left(\xi + G_{s}
\right) d s, \quad t \geq 0.
$$
As a consequence, the unique continuous solution  $X$ that solves
\eqref{linear-GOU} is given by
\begin{eqnarray*}
X_t&=&-\theta e^{-\theta t} \int_{0}^{t} e^{\theta s}\left(\xi +
G_{s} \right) d s+\xi + G_{t}\\
&=&e^{-\theta t}\left(\xi+ \int_{0}^{t} e^{\theta s} d G_{s}
\right), \quad t \geq 0,
\end{eqnarray*}
where the latter equality comes from \eqref{IBP-G}. This completes
the proof.
\end{proof}

In what follows   $R_G(s,t):=E\left(G_sG_t\right),\ s,t\in\R,$
denotes the covariance function of the process $G$.

\begin{lemma}\label{lemma-cov}Let $\theta\in \R$ and $-\infty<s<t\leq u<v<\infty$. Assume that $(\mathcal{A})$ holds and
 $\frac{\partial^2R_G}{\partial y\partial x}$ is a continuous function on $\R^2 \setminus\{(x,y)|x=y\}$.  Then
\begin{equation}
E\left(\int_s^te^{\theta x}dG_x\int_u^ve^{\theta
y}dG_y\right)=\int_u^v\int_s^te^{\theta x}e^{\theta
y}\frac{\partial^2R_G}{\partial y\partial x}(x,y)
dxdy,\label{cov-wienner-integ}
\end{equation}
 provided that the integral on the right-hand side converges.\\
In addition, if we assume $\theta>0$,  the identity
\eqref{cov-wienner-integ} is also valid for   $s=-\infty$. In other
words, for every $\theta>0$ and $-\infty<t\leq u<v<\infty$,
\begin{equation}
E\left(\int_{-\infty}^te^{\theta x}dG_x\int_u^ve^{\theta
y}dG_y\right)=\int_u^v\int_{-\infty}^te^{\theta x}e^{\theta
y}\frac{\partial^2R_G}{\partial y\partial x}(x,y)
dxdy,\label{cov-wienner-integ-infty}
\end{equation}
 provided that the integral on the right-hand side converges.
\end{lemma}

\begin{proof}Let us first  suppose  $t= u$. Using \eqref{IBP-G}, we have
\begin{eqnarray*}
&&E\left(\int_s^te^{\theta x}dG_x\int_t^ve^{\theta
y}dG_y\right)\\&=&e^{\theta t}e^{\theta v}R_G(t,v)-e^{2\theta t}
R_G(t,t)-e^{\theta s}e^{\theta v}R_G(s,v)+e^{\theta s}e^{\theta
t}R_G(s,t)\\&&-\theta e^{\theta t}\int_t^ve^{\theta y}R_G(t,y)
dy+\theta e^{\theta s}\int_t^ve^{\theta y}R_G(s,y) dy-\theta
e^{\theta v}\int_s^te^{\theta x}R_G(x,v) dx\\&&+\theta e^{\theta
t}\int_s^te^{\theta x}R_G(x,t) dx+\theta^2\int_t^v\int_s^te^{\theta
x}e^{\theta y}R_G(x,y)dxdy.
\end{eqnarray*}
On the other hand, by integrating by parts, we have
\begin{eqnarray*}
-\theta e^{\theta v} \int_s^te^{\theta x}R_G(x,v) dx=-e^{\theta
v}e^{\theta t}R_G(t,v)+e^{\theta v}e^{\theta s}R_G(s,v)+e^{\theta
v}\int_s^te^{\theta x}\frac{\partial R_G}{\partial x}(x,v) dx,
\end{eqnarray*}
\begin{eqnarray*}
\theta e^{\theta t}\int_s^te^{\theta x}R_G(x,t) dx=e^{2\theta
t}R_G(t,t)-e^{\theta t}e^{\theta s}R_G(s,t)-e^{\theta
t}\int_s^te^{\theta x}\frac{\partial R_G}{\partial x}(x,t) dx,
\end{eqnarray*}
and
\begin{eqnarray*}\theta^2\int_t^v\int_s^te^{\theta
x}e^{\theta y}R_G(x,y)dxdy&=&\theta\int_t^ve^{\theta y}e^{\theta
t}R_G(t,y)dy-\theta\int_t^ve^{\theta y}e^{\theta
s}R_G(s,y)dy\\&&-\theta\int_t^v\int_s^te^{\theta x}e^{\theta
y}\frac{\partial R_G}{\partial x}(x,y) dxdy.
\end{eqnarray*}
These equalities imply
\begin{eqnarray*}
E\left(\int_s^te^{\theta x}dG_x\int_t^ve^{\theta
y}dG_y\right)&=&e^{\theta v}\int_s^te^{\theta x}\frac{\partial
R_G}{\partial x}(x,v) dx-e^{\theta t}\int_s^te^{\theta
x}\frac{\partial R_G}{\partial x}(x,t)
dx\\&&-\theta\int_t^v\int_s^te^{\theta x}e^{\theta y}\frac{\partial
R_G}{\partial x}(x,y) dxdy.
\end{eqnarray*}
Further,
\begin{eqnarray*}
-\theta\int_t^v\int_s^te^{\theta x}e^{\theta y}\frac{\partial
R_G}{\partial x}(x,y) dxdy&=&-e^{\theta v}\int_s^te^{\theta
x}\frac{\partial R_G}{\partial x}(x,v) dx+e^{\theta
t}\int_s^te^{\theta x}\frac{\partial R_G}{\partial x}(x,t)
dx\\&&+\int_t^v\int_s^te^{\theta x}e^{\theta
y}\frac{\partial^2R_G}{\partial y\partial x}(x,y) dxdy,
\end{eqnarray*}
which proves \eqref{cov-wienner-integ} for $t= u$.\\
Let us now  suppose  $t< u$. From above  we deduce that
\begin{eqnarray*}
E\left(\int_s^te^{\theta x}dG_x\int_u^ve^{\theta
y}dG_y\right)&=&E\left(\int_s^te^{\theta x}dG_x\int_t^ve^{\theta
y}dG_y\right)-E\left(\int_s^te^{\theta x}dG_x\int_t^ue^{\theta
y}dG_y\right)\\&=&\int_t^v\int_s^te^{\theta x}e^{\theta
y}\frac{\partial^2R_G}{\partial y\partial x}(x,y)
dxdy-\int_t^u\int_s^te^{\theta x}e^{\theta
y}\frac{\partial^2R_G}{\partial y\partial x}(x,y) dxdy\\&=&
\int_u^v\int_s^te^{\theta x}e^{\theta
y}\frac{\partial^2R_G}{\partial y\partial x}(x,y) dxdy,
\end{eqnarray*}
which completes the proof of \eqref{cov-wienner-integ}.\\
 Finally, using \eqref{IBP-G-infty} and following the same arguments as above, the claim \eqref{cov-wienner-integ-infty} follows.
 \end{proof}
 \begin{remark}Lemma \ref{lemma-cov} was proved in \cite[Lemma 2.1]{CKM} in
 the case when $G=B^H$ is a fBm with Hurst
 parameter $H\in (0, \frac12) \cup(\frac12, 1]$.
 \end{remark}

 \begin{lemma}\label{lemma-tec1}Let  $\rho$ be a positive measurable function on $\R$. Then, for all
 $0<s<t$,
 \begin{eqnarray*}
e^{-\theta t}e^{-\theta s}\int_s^t\int_0^se^{\theta x}e^{\theta
y}\rho(y-x)dxdy&\leq&e^{-\theta
(t-s)}\int_0^{t-s}\int_{-\infty}^0e^{\theta x}e^{\theta
y}\rho(y-x)dxdy,
\end{eqnarray*}
provided that the integral on the right-hand side converges.
\end{lemma}
\begin{proof}By change of variables $u=x-s$ and $v=y-s$, we obtain
\begin{eqnarray*}
e^{-\theta t}e^{-\theta s}\int_s^t\int_0^se^{\theta x}e^{\theta
y}\rho(y-x)dxdy&=&e^{-\theta (t-s)}\int_0^{t-s}\int_{-s}^0e^{\theta
u}e^{\theta v}\rho(v-u)dudv\\&\leq&e^{-\theta
(t-s)}\int_0^{t-s}\int_{-\infty}^0e^{\theta u}e^{\theta
v}\rho(v-u)dudv,
\end{eqnarray*}
 which finishes the proof.
\end{proof}

 \begin{lemma}Let $\gamma\in(0,\frac12)\cup(\frac12,1)$ and $\theta>0$. Then,  as $t\rightarrow\infty$,
 \begin{eqnarray}
e^{-\theta t}\int_0^t\int_{-\infty}^0e^{\theta x}e^{\theta
y}(y-x)^{2\gamma-2}dxdy\sim
\frac{t^{2\gamma-2}}{\theta^2}.\label{equi1}
\end{eqnarray}
As a consequence,  there exists a constant $C>0$ that depends only
on $\theta$ and $\gamma$ such that for every   $t-s>2$,
  \begin{eqnarray}
e^{-\theta t}e^{-\theta s}\int_s^t\int_0^se^{\theta x}e^{\theta
y}(y-x)^{2\gamma-2}dxdy&\leq&C (t-s)^{2\gamma-2}.\label{ineq1}
\end{eqnarray}
  \end{lemma}
  \begin{proof}First we prove \eqref{equi1}. Let $t>2$.
  Making the change of variables $u=y-x$, we get
\begin{eqnarray}
e^{-\theta t}\int_0^t\int_{-\infty}^0e^{\theta x}e^{\theta
y}(y-x)^{2\gamma-2}dxdy&=&e^{-\theta
t}\int_0^t\int_y^{\infty}e^{-\theta u}e^{2\theta
y}u^{2\gamma-2}dudy\nonumber\\
&=&e^{-\theta t}\int_0^{\infty}due^{-\theta
u}u^{2\gamma-2}\int_{0}^{u\wedge t}dye^{2\theta y}\nonumber
\\
&=&\frac{e^{-\theta t}}{2\theta}\int_0^{\infty}e^{-\theta
u}u^{2\gamma-2}\left(e^{2\theta (u\wedge t)}-1\right)du\nonumber
\\
&=&\frac{e^{-\theta t}}{2\theta}\int_0^{t}e^{-\theta
u}u^{2\gamma-2}\left(e^{2\theta u}-1\right)du\nonumber
\\
&&+\frac{e^{-\theta t}}{2\theta}\int_t^{\infty}e^{-\theta
u}u^{2\gamma-2}\left(e^{2\theta t}-1\right)du\nonumber
\\
&=:&A_{1,t}+A_{2,t}.\label{A}
\end{eqnarray}
Furthermore,
\begin{eqnarray}
A_{1,t}&=&\frac{e^{-\theta t}}{2\theta}\int_0^{t}e^{-\theta
u}u^{2\gamma-2}\left(e^{2\theta u}-1\right)du\nonumber\\
&=&\frac{e^{-\theta t}}{2\theta}\int_0^{1}e^{-\theta
u}u^{2\gamma-2}\left(e^{2\theta u}-1\right)du +\frac{e^{-\theta
t}}{2\theta}\int_1^{t}e^{-\theta u}u^{2\gamma-2}\left(e^{2\theta
u}-1\right)du.\label{A1_ineq1}
\end{eqnarray}
Since $\frac{e^{2\theta u}-1}{u}\longrightarrow2\theta$  as
$u\rightarrow0$, then the function $\frac{e^{2\theta u}-1}{u}$ is
bounded on (0,1], that is, $\sup_{u\in(0,1]}\frac{e^{2\theta
u}-1}{u}<C<\infty$. This implies that
\begin{eqnarray}
\int_0^{1}e^{-\theta u}u^{2\gamma-2}\left(e^{2\theta
u}-1\right)du&\leq&C\int_0^{1}e^{-\theta u}u^{2\gamma-1}du\nonumber\\
&\leq&C\int_0^{1}u^{2\gamma-1}du
\nonumber\\
&=&\frac{C}{2\gamma}.\label{A1_ineq2}
\end{eqnarray}
 On the other hand, as $t\rightarrow\infty$,
\begin{eqnarray}
\frac{e^{-\theta t}}{2\theta}\int_1^{t}e^{-\theta
u}u^{2\gamma-2}\left(e^{2\theta u}-1\right)du &=&\frac{e^{-\theta
t}}{2\theta}\int_1^{t}e^{\theta u}u^{2\gamma-2} du-\frac{e^{-\theta
t}}{2\theta}\int_1^{t}e^{-\theta u}u^{2\gamma-2} du\nonumber
\\&\sim&\frac{t^{2\gamma-2}}{2\theta^2},\label{A1_ineq3}
\end{eqnarray}where we used the fact that
$e^{-\theta t}\int_1^{t}e^{-\theta u}u^{2\gamma-2} du\leq
Ce^{-\theta t}$, and by   L'H\^opital's Rule, we have
\begin{eqnarray*}\lim_{t\rightarrow\infty}\frac{\int_1^{t}e^{\theta u}u^{2\gamma-2} du}{2\theta t^{2\gamma-2} e^{\theta
t}}&=&\lim_{t\rightarrow\infty}\frac{e^{\theta
t}t^{2\gamma-2}}{2\theta t^{2\gamma-2} e^{\theta
t}\left(\theta+(2\gamma-2)t^{-1}\right)
}=\lim_{t\rightarrow\infty}\frac{1}{2\theta
\left(\theta+(2\gamma-2)t^{-1}\right) }=\frac{1}{2\theta^2}.
\end{eqnarray*}
 Thus, combining \eqref{A1_ineq1}, \eqref{A1_ineq2} and \eqref{A1_ineq3}, we deduce that, as $t\rightarrow\infty$,
\begin{eqnarray}
A_{1,t}&\sim&\frac{t^{2\gamma-2}}{2\theta^2}.\label{A1}
\end{eqnarray}
For $A_{2,t}$, we have, as $t\rightarrow\infty$,
\begin{eqnarray}
A_{2,t}&=&\frac{e^{-\theta t}}{2\theta}\int_t^{\infty}e^{-\theta
u}u^{2\gamma-2}\left(e^{2\theta t}-1\right)du\nonumber\\
&=&\frac{e^{\theta t}}{2\theta}\int_t^{\infty}e^{-\theta
u}u^{2\gamma-2}du-\frac{e^{-\theta
t}}{2\theta}\int_t^{\infty}e^{-\theta u}u^{2\gamma-2}du\nonumber
\\&\sim&\frac{t^{2\gamma-2}}{2\theta^2},\label{A2}
\end{eqnarray}where we used the fact that
$\frac{e^{-\theta t}}{2\theta}\int_t^{\infty}e^{-\theta
u}u^{2\gamma-2}du\leq Ct^{2\gamma-2}e^{-\theta t}$, and by using
L'H\^opital's Rule, we have
\begin{eqnarray*}\lim_{t\rightarrow\infty}\frac{\int_t^{\infty}e^{-\theta
u}u^{2\gamma-2}du}{2\theta t^{2\gamma-2} e^{-\theta
t}}&=&\lim_{t\rightarrow\infty}\frac{-e^{-\theta
t}t^{2\gamma-2}}{2\theta t^{2\gamma-2} e^{-\theta
t}\left(-\theta+(2\gamma-2)t^{-1}\right)
}\\
&=&\lim_{t\rightarrow\infty}\frac{-1}{2\theta
\left(-\theta+(2\gamma-2)t^{-1}\right) }=\frac{1}{2\theta^2}.
\end{eqnarray*}
Therefore, \eqref{A}, \eqref{A1} and \eqref{A2} prove \eqref{equi1}.\\
The estimate \eqref{ineq1} is a direct consequence of \eqref{equi1}
and Lemma \ref{lemma-tec1}.
\end{proof}

 \begin{lemma}Assume that $\gamma\in(0,1)$ and $\theta>0$. Then,   as $t\rightarrow\infty$,
\begin{eqnarray}
&&e^{-\theta t}\int_0^t\int_{-\infty}^0e^{\theta x}e^{\theta
y}\left(e^{\frac{y-x}{2\gamma}}\pm e^{-\frac{y-x}{2\gamma}}\right)^{2\gamma-2}dxdy\nonumber\\
&\sim& e^{-\min \left(\theta,\frac{1}{\gamma}-1\right) t} \times
\left\{
\begin{array}{ll} \int_0^{\infty}\left(e^{\theta u}-e^{-\theta
u}\right)\left(e^{\frac{u}{2\gamma}}\pm e^{-\frac{u}{2\gamma}}\right)^{2\gamma-2}du & \text { if }\ \theta<\frac{1}{\gamma}-1, \\
\frac{t}{2\theta} & \text { if }\ \theta=\frac{1}{\gamma}-1, \\
\frac{1}{\theta^2-\left(\frac{1}{\gamma}-1\right)^2} & \text { if }\
\theta>\frac{1}{\gamma}-1, \end{array} \right.\label{equi2}
\end{eqnarray}
and
\begin{eqnarray}
&&e^{-\theta t}\int_0^t\int_{-\infty}^0e^{\theta x}e^{\theta
y}\left[\left(e^{\frac{y-x}{2\gamma}}+
e^{-\frac{y-x}{2\gamma}}\right)^{2\gamma-2}
-\left(e^{\frac{y-x}{2\gamma}}-
e^{-\frac{y-x}{2\gamma}}\right)^{2\gamma-2}\right]dxdy \sim e^{-\min
\left(\theta,\frac{2}{\gamma}-1\right) t}\nonumber
\\&& \times \left\{
\begin{array}{ll} \int_0^{\infty}\left(e^{\theta u}-e^{-\theta
u}\right)\left[\left(e^{\frac{u}{2\gamma}}+
e^{-\frac{u}{2\gamma}}\right)^{2\gamma-2}
-\left(e^{\frac{u}{2\gamma}}- e^{-\frac{u}{2\gamma}}\right)^{2\gamma-2}\right]du & \text { if } \theta<\frac{2}{\gamma}-1, \\
\frac{(2\gamma-2)te^{\frac{t}{\gamma}}}{\theta} & \text { if } \theta=\frac{2}{\gamma}-1, \\
\frac{4\gamma-4}{\theta^2-\left(\frac{2}{\gamma}-1\right)^2} & \text
{ if } \theta>\frac{2}{\gamma}-1. \end{array} \right.\label{equi2ii}
\end{eqnarray}
Consequently,  there exists a constant $C>0$ that depends only on
$\theta$  and $\gamma$  such that for every   $t-s>2$,
\begin{eqnarray}
e^{-\theta t}e^{-\theta s}\int_s^t\int_0^se^{\theta x}e^{\theta
y}\left(e^{\frac{y-x}{2\gamma}}\pm
e^{-\frac{y-x}{2\gamma}}\right)^{2\gamma-2}dxdy&\leq&C\left\{
\begin{array}{ll} e^{-\min \left( \theta,\frac{1}{\gamma}-1\right)
|t-s|}  & \text { if }\ \theta\neq \frac{1}{\gamma}-1, \\
te^{-\min \left( \theta,\frac{1}{\gamma}-1\right) |t-s|} & \text {
if }\ \theta=\frac{1}{\gamma}-1,
\end{array} \right.\label{ineq2}
\end{eqnarray}
and
\begin{eqnarray}
&&e^{-\theta t}e^{-\theta s}\int_s^t\int_0^se^{\theta x}e^{\theta
y}\left[\left(e^{\frac{y-x}{2\gamma}}+
e^{-\frac{y-x}{2\gamma}}\right)^{2\gamma-2}
-\left(e^{\frac{y-x}{2\gamma}}-
e^{-\frac{y-x}{2\gamma}}\right)^{2\gamma-2}\right]dxdy\nonumber\\&&\quad
\leq C\left\{
\begin{array}{ll}  e^{-\min \left( \theta,\frac{2}{\gamma}-1\right)
|t-s|}  & \text { if }\ \theta\neq \frac{2}{\gamma}-1, \\
te^{-\min \left( \theta,\frac{1}{\gamma}-1\right) |t-s|} & \text {
if }\ \theta=\frac{2}{\gamma}-1.
\end{array} \right.\label{ineq2ii}
\end{eqnarray}
  \end{lemma}

\begin{proof}
 Let us  prove \eqref{equi2}. Let $t>2$.
  Making the change of variables $u=y-x$, we get
\begin{eqnarray}
&&e^{-\theta t}\int_0^t\int_{-\infty}^0e^{\theta x}e^{\theta
y}\left(e^{\frac{y-x}{2\gamma}}\pm
e^{-\frac{y-x}{2\gamma}}\right)^{2\gamma-2}dxdy\nonumber\\&=&e^{-\theta
t}\int_0^t\int_y^{\infty}e^{-\theta u}e^{2\theta
y}\left(e^{\frac{u}{2\gamma}}\pm e^{-\frac{u}{2\gamma}}\right)^{2\gamma-2}dudy\nonumber\\
&=&\frac{e^{-\theta t}}{2\theta}\int_0^{t}e^{-\theta
u}\left(e^{\frac{u}{2\gamma}}\pm
e^{-\frac{u}{2\gamma}}\right)^{2\gamma-2}\left(e^{2\theta
u}-1\right)du +\frac{e^{-\theta
t}}{2\theta}\int_t^{\infty}e^{-\theta
u}\left(e^{\frac{u}{2\gamma}}\pm
e^{-\frac{u}{2\gamma}}\right)^{2\gamma-2}\left(e^{2\theta
t}-1\right)du\nonumber
\\
&=:&B_{1,t}+B_{2,t}.\label{B}
\end{eqnarray}
Further,
\begin{eqnarray*}
B_{1,t}&=&\frac{e^{-\theta t}}{2\theta}\int_0^{t}e^{-\theta
u}\left(e^{\frac{u}{2\gamma}}\pm e^{-\frac{u}{2\gamma}}\right)^{2\gamma-2}\left(e^{2\theta u}-1\right)du\\
&=&\frac{e^{-\theta t}}{2\theta}\int_0^{1}e^{-\theta
u}\left(e^{\frac{u}{2\gamma}}\pm
e^{-\frac{u}{2\gamma}}\right)^{2\gamma-2}\left(e^{2\theta
u}-1\right)du +\frac{e^{-\theta t}}{2\theta}\int_1^{t}e^{-\theta
u}\left(e^{\frac{u}{2\gamma}}\pm
e^{-\frac{u}{2\gamma}}\right)^{2\gamma-2}\left(e^{2\theta
u}-1\right)du.
\end{eqnarray*}
Since $\frac{e^{2\theta u}-1}{u}\longrightarrow2\theta$,
$\frac{e^{\frac{u}{2\gamma}}-e^{-\frac{u}{2\gamma}}}{u}\rightarrow\frac{1}{\gamma}$
and $e^{\frac{u}{2\gamma}}+e^{-\frac{u}{2\gamma}}\rightarrow2$ as
$u\rightarrow0$, then
 \[\int_0^{1}e^{-\theta
u}\left(e^{\frac{u}{2\gamma}}\pm
e^{-\frac{u}{2\gamma}}\right)^{2\gamma-2}\left(e^{2\theta
u}-1\right)du\leq
C\int_0^{1}\left(u+u^{2\gamma-1}\right)du<\infty.\]
 On the other hand,
 \begin{eqnarray}e^{-\theta
t}\left(e^{\frac{t}{2\gamma}}\pm
e^{-\frac{t}{2\gamma}}\right)^{2\gamma-2}\left(e^{2\theta
t}-1\right)\sim e^{-\left(\frac{1}{\gamma}-1-\theta\right) t}\quad
\mbox{ as } t\rightarrow\infty.\label{equi-integrand}\end{eqnarray}
 Thus, for $\theta<\frac{1}{\gamma}-1$, we obtain, using \eqref{equi-integrand}, $\displaystyle\int_1^{\infty}e^{-\theta
u}\left(e^{\frac{u}{2\gamma}}\pm
e^{-\frac{u}{2\gamma}}\right)^{2\gamma-2}\left(e^{2\theta
u}-1\right)du<\infty.$ In this case, we have as
$t\rightarrow\infty$,
\begin{eqnarray*}
\frac{e^{-\theta t}}{2\theta}\int_1^{t}e^{-\theta
u}\left(e^{\frac{u}{2\gamma}}\pm
e^{-\frac{u}{2\gamma}}\right)^{2\gamma-2}\left(e^{2\theta
u}-1\right)du &\sim&\frac{e^{-\theta
t}}{2\theta}\int_1^{\infty}e^{-\theta
u}\left(e^{\frac{u}{2\gamma}}\pm
e^{-\frac{u}{2\gamma}}\right)^{2\gamma-2}\left(e^{2\theta
u}-1\right)du.
\end{eqnarray*}
 For $\theta>\frac{1}{\gamma}-1$,
  since $0<\left(1\pm e^{-\frac{u}{\gamma}}\right)^{2\gamma-2}\rightarrow1$
  as $u\rightarrow\infty$,  there exists $C>0$ such that \[\int_1^{t}e^{-\theta
u}\left(e^{\frac{u}{2\gamma}}\pm
e^{-\frac{u}{2\gamma}}\right)^{2\gamma-2}\left(e^{2\theta
u}-1\right)du\geq C\int_1^{t}e^{\left(\theta-\frac{1}{\gamma}
+1\right)u}\left(1-e^{-2\theta }\right)du\rightarrow\infty\] as
$t\rightarrow\infty$. Combining this together with
\eqref{equi-integrand} and L'H\^opital's Rule, we
 get
\begin{eqnarray*}&&\lim_{t\rightarrow\infty}\frac{\int_1^{t}e^{-\theta
u}\left(e^{\frac{u}{2\gamma}}\pm
e^{-\frac{u}{2\gamma}}\right)^{2\gamma-2}\left(e^{2\theta
u}-1\right)du}{2\theta e^{\theta
t}e^{-\left(\frac{1}{\gamma}-1\right) t}}
\\&=&\lim_{t\rightarrow\infty}\frac{\left(e^{\frac{t}{2\gamma}}\pm e^{-\frac{t}{2\gamma}}\right)^{2\gamma-2}\left(1-e^{-2\theta
t}\right)}{2\theta e^{-\left(\frac{1}{\gamma}-1\right)
t}\left(\theta-\frac{1}{\gamma}+1\right)}
\\&=&\frac{1}{2\theta \left(\theta-\frac{1}{\gamma}+1\right)}.
\end{eqnarray*}
In this situation, we have as $t\rightarrow\infty$,
\begin{eqnarray*}
\frac{e^{-\theta t}}{2\theta}\int_1^{t}e^{-\theta
u}\left(e^{\frac{u}{2\gamma}}\pm
e^{-\frac{u}{2\gamma}}\right)^{2\gamma-2}\left(e^{2\theta
u}-1\right)du &\sim&\frac{e^{-\left(\frac{1}{\gamma}-1\right)
t}}{2\theta \left(\theta-\frac{1}{\gamma}+1\right)}.
\end{eqnarray*}
For $\theta=\frac{1}{\gamma}-1$, since $0<\left(1\pm
e^{-\frac{u}{\gamma}}\right)^{2\gamma-2}\rightarrow1$
  as $u\rightarrow\infty$,  there exists $C>0$ such that, as $t\rightarrow\infty$,
\[\int_1^{t}e^{-\theta
u}\left(e^{\frac{u}{2\gamma}}\pm
e^{-\frac{u}{2\gamma}}\right)^{2\gamma-2}\left(e^{2\theta
u}-1\right)du\geq C
\int_1^{t}\left(1-e^{-2\theta}\right)du\rightarrow\infty.\]
Combining this together with \eqref{equi-integrand} and
L'H\^opital's Rule leads to
\begin{eqnarray*}\lim_{t\rightarrow\infty}
\frac{\int_1^{t}e^{-\theta u}\left(e^{\frac{u}{2\gamma}}\pm
e^{-\frac{u}{2\gamma}}\right)^{2\gamma-2}\left(e^{2\theta
u}-1\right)du}{t} &=&\lim_{t\rightarrow\infty} \left(1\pm
e^{-\frac{t}{2\gamma}}\right)^{2\gamma-2}\left(1-e^{-2\theta
t}\right)\\&=&1.
\end{eqnarray*}
In this case, we get as $t\rightarrow\infty$,
\begin{eqnarray*}
\frac{e^{-\theta t}}{2\theta}\int_1^{t}e^{-\theta
u}\left(e^{\frac{u}{2\gamma}}\pm
e^{-\frac{u}{2\gamma}}\right)^{2\gamma-2}\left(e^{2\theta
u}-1\right)du &\sim&\frac{te^{-\theta t}}{2\theta}.
\end{eqnarray*}
Thus,  as $t\rightarrow\infty$,
\begin{eqnarray}
B_{1,t}&\sim& e^{-\min \left(\theta,\frac{1}{\gamma}-1\right) t}
\times \left\{
\begin{array}{ll} \int_0^{\infty}\left(e^{\theta u}-e^{-\theta
u}\right)\left(e^{\frac{u}{2\gamma}}\pm e^{-\frac{u}{2\gamma}}\right)^{2\gamma-2}du & \text { if }\ \theta<\frac{1}{\gamma}-1, \\
\frac{t}{2\theta} & \text { if }\ \theta=\frac{1}{\gamma}-1, \\
\frac{1}{2\theta\left(\theta-\frac{1}{\gamma}+1\right)} & \text { if
}\ \theta>\frac{1}{\gamma}-1. \end{array} \right.\label{B1}
\end{eqnarray}
For $B_{2,t}$, we have
\begin{eqnarray*} B_{2,t}&=&\frac{e^{-\theta t}}{2\theta}\int_t^{\infty}e^{-\theta
u}\left(e^{\frac{u}{2\gamma}}\pm
e^{-\frac{u}{2\gamma}}\right)^{2\gamma-2}\left(e^{2\theta
t}-1\right)du\\&=&\frac{e^{\theta t}\left(1-e^{-2\theta
t}\right)}{2\theta}\int_t^{\infty}e^{-\theta
u}e^{-\left(\frac{1}{\gamma}-1\right) u}\left(1\pm
e^{-\frac{u}{\gamma}}\right)^{2\gamma-2}du,
\end{eqnarray*}
where, as $t\rightarrow\infty$,
\begin{eqnarray*}e^{\theta
t}\int_t^{\infty}e^{-\theta u}e^{-\left(\frac{1}{\gamma}-1\right)
u}\left(1\pm e^{-\frac{u}{\gamma}}\right)^{2\gamma-2}du \sim
\frac{e^{-\left(\frac{1}{\gamma}-1\right)
t}}{\theta+\frac{1}{\gamma}-1}
\end{eqnarray*}
since, by  L'H\^opital's Rule, we have
\begin{eqnarray*}\lim_{t\rightarrow\infty}\frac{\int_t^{\infty}e^{-\theta
u}e^{-\left(\frac{1}{\gamma}-1\right) u}\left(1\pm
e^{-\frac{u}{\gamma}}\right)^{2\gamma-2}du}{e^{-\theta
t}e^{-\left(\frac{1}{\gamma}-1\right) t}}
&=&\lim_{t\rightarrow\infty}\frac{1}{\theta+\frac{1}{\gamma}-1}\left(1\pm e^{-\frac{t}{\gamma}}\right)^{2\gamma-2}\\
&=&\frac{1}{\theta+\frac{1}{\gamma}-1}.
\end{eqnarray*}
Hence, as $t\rightarrow\infty$, \begin{eqnarray} B_{2,t} \sim
\frac{e^{-\left(\frac{1}{\gamma}-1\right)
t}}{2\theta\left(\theta+\frac{1}{\gamma}-1\right)}.\label{B2}
\end{eqnarray}
Then, using \eqref{B}, \eqref{B1} and \eqref{B2}, we obtain  \eqref{equi2}.\\
By similar arguments as above and the fact
\[\left(e^{\theta u}-e^{-\theta
u}\right)\left[\left(e^{\frac{u}{2\gamma}}+
e^{-\frac{u}{2\gamma}}\right)^{2\gamma-2}
-\left(e^{\frac{u}{2\gamma}}-
e^{-\frac{u}{2\gamma}}\right)^{2\gamma-2}\right]\sim (4\gamma-4)
e^{-\left(\frac{2}{\gamma}-1-\theta\right) t} \quad \mbox{ as }
t\rightarrow\infty,\] the claim
\eqref{equi2ii} follows.\\
 The estimate \eqref{ineq2} and \eqref{ineq2ii} are direct
consequences of \eqref{equi2} and \eqref{equi2ii}, respectively,
combined with Lemma \ref{lemma-tec1}.
  \end{proof}

Consider the  Ornstein-Uhlenbeck process $X:=\{X_{t}, t\geq 0\}$
defined by the following linear stochastic differential equation
\begin{equation}
X_{0}=0,\qquad dX_{t}=-\theta X_{t}dt+dG_{t}, \quad t\geq0,
\label{GOU}
\end{equation}
where the process $G$ satisfies $(\mathcal{A})$ and $\theta\in\R$.\\
According to Lemma \ref{lemma-Riemann integral}, the solution of the
equation \eqref{GOU} can be expressed explicitly as
\begin{equation}
 X_t = e^{-\theta t} \int_0^te^{\theta r}dG_r=G_t-\theta e^{-\theta t}\int_0^te^{\theta r}G_rdr. \label{decomp-GOU}
\end{equation}
This implies that
\begin{equation}
 E\left(X_t^2\right) = R_G(t,t)-2\theta e^{-\theta t}\int_0^te^{\theta r}R_G(r,t)dr
 +\theta^2 e^{-2\theta t}\int_0^t\int_0^te^{\theta r}e^{\theta s}R_G(r,s)drds. \label{decomp-var-GOU}
\end{equation}
If $\theta>0$, then,  according to \eqref{IBP-G-infty},
\begin{equation}
Z_{t}=\int_{-\infty }^{t}e^{-\theta (t-s)}dG_{s}\label{Z}
\end{equation}
is well defined as a Riemann-Stieltjes integral, so we can write
\begin{equation}
 X_t = Z_{t}-e^{-\theta t} Z_{0},\quad t\geq0. \label{link-X-Z}
\end{equation}

We will also make use of the following lemmas.

\begin{lemma}[\cite{EEO}]\label{lemma key1 of applications}Let $g : [0,\infty)\times[0,\infty)\longrightarrow\mathbb{R}$
 be a symmetric function such that $\frac{\partial  g }{\partial s }(s,r)$
 and $\frac{\partial^2 g }{\partial s\partial r}(s,r)$ are integrable on   $(0,t)\times[0,t)$ for all $t>0$. Then, for every $t>0$,
\begin{eqnarray}  \Delta_g(t)&:=&g(t,t)-2\theta e^{-\theta t}\int_0^t g(s,t)e^{\theta s}ds
+\theta^2 e^{-2\theta t}\int_0^t\int_0^t g(s,r)e^{\theta (s+r)}dr ds
\nonumber\\&=&2e^{-2\theta t}\int_0^te^{\theta s} \frac{\partial g
}{\partial s}(s,0)ds+2 e^{-2\theta t}\int_0^tdse^{\theta s}\int_0^s
dr\frac{\partial^2 g }{\partial s\partial r}(s,r)e^{\theta
r}.\label{key1}
\end{eqnarray}
\end{lemma}

The following lemma is an immediate consequence of
\eqref{decomp-GOU} and \eqref{cov-wienner-integ}.
\begin{lemma}[\cite{DEV}]
\label{calculcov} Let $G$ be a measurable process satisfying  $(\mathcal{A})$ and $X$ is the solution of the equation (\ref%
{GOU}). Assume that $\frac{\partial^2R_G}{\partial y\partial x}$ is
continuous on $\R^2 \setminus\{(x,y)|x=y\}$.
 Then,
for every $0 < s < t$, we have
\begin{eqnarray}
E\left(X_{s}X_{t}\right) = e^{- \theta(t-s)} E\left(X_{s}^{2}\right)
+ e^{- \theta t} e^{-
\theta s} \int_{s}^{t} e^{\theta v} \int_{0}^{s} e^{\theta u} \frac{%
\partial^2R_{G}}{\partial u\partial v} (u,v) du dv,  \label{cov of X GOU}
\end{eqnarray}
provided that the integral on the right-hand side converges.
\end{lemma}

\section{Langevin equations driven by noises  with
stationary increments}

 Recall that a
process $G=\{G_{t}, t \in \mathbb{R}\}$ has stationary increments
if, for all $s \in \mathbb{R},\{G_{t}-G_{0}, t \in \mathbb{R}\}$ has
the same finite distributions as $\{G_{t+s}-G_{s}, t \in
\mathbb{R}\}$.

The existence and uniqueness of stationary solutions to Langevin
equations driven by noise processes with stationary increments are
discussed in the following theorem.

\begin{theorem}\label{thm-stat-incr} Let $\{G_{t}, t \in \mathbb{R}\}$ be a measurable process with stationary increments satisfying  $(\mathcal{A})$.
Assume $\theta>0$. Then,
\begin{enumerate}
\item[(a)] The solution  $Z_{t}=\int_{-\infty }^{t}e^{-\theta (t-s)}dG_{s},\ t\geq0$, of the equation
\begin{eqnarray}d Z_{t}=-\theta Z_{t} dt+d G_{t},\ Z_0=\int_{-\infty }^{0}e^{\theta
s}dG_{s},\quad t\geq0,\label{Z-OU}\end{eqnarray} is a  stationary
process.
\item[(b)] In addition, if we assume that  the
function $\rho_G(t):=E\left(G_t^2\right),\, t\in\R$, is twice
continuously differentiable on $\R\setminus\{0\}$, then
\begin{eqnarray}E\left(Z_tZ_0\right)=e^{-\theta t}E\left(Z_0^2\right)+\frac{e^{-\theta t}}{2}\int_{0}^{t}\int_{-\infty}^{0}
e^{\theta u} e^{\theta v}\rho_G''(v-u)dudv,\quad
t\geq0,\label{cov-Z-stat-incr}
\end{eqnarray}
and the process $X$, given by \eqref{decomp-GOU}, satisfies
\begin{eqnarray}
E\left(X_{s}X_{t}\right) = e^{- \theta(t-s)} E\left(X_{s}^{2}\right)
+ e^{- \theta t} e^{- \theta s} \int_{s}^{t} e^{\theta v}
\int_{0}^{s} e^{\theta u} \rho_G''(v-u) du dv, \quad s,t\geq0,
\label{cov of X-stat-incr}
\end{eqnarray}
\begin{eqnarray}\left|E\left(X_t^2\right)-E\left(Z_0^2\right)\right|\leq
Ce^{-\theta t},\quad t\geq0.\label{var of X-stat-incr}\end{eqnarray}
Also,
\begin{eqnarray}E\left(Z_0^2\right)=\frac{\theta}{2}\int_{0}^{\infty}  e^{-\theta
t}\rho_G(t)dt,\label{var-Z-stat-incr}
\end{eqnarray}
provided that the integrals above converge.
 \end{enumerate}
\end{theorem}
\begin{proof} Applying Lemma \ref{lemma-Riemann integral} for $\xi=\int_{-\infty }^{0}e^{\theta
s}dG_{s}$, the unique solution to the equation \eqref{Z-OU} can be
expressed as
$$
Z_{t}=\int_{-\infty }^{t}e^{-\theta (t-s)}dG_{s}=G_{t}-\theta
e^{-\theta t} \int_{-\infty}^{t} e^{\theta s} G_{s} d s, \quad t
\geq0.
$$
On the other hand, it follows from  \cite[Theorem 2.1]{BB} that
 the process \[G_{t}-\theta e^{-\theta t} \int_{-\infty}^{t} e^{\theta s}
G_{s} d s,\ t\geq0,\]  is a unique-in-law stationary solution to the
Langevin equation \eqref{Z-OU}. Thus the part \emph{(a)} is proved.
\\
 Let us prove  the part \emph{(b)}. Since $G_0=0$ a.s. and $G$ has stationary
 increments, we have
\begin{eqnarray}R_G(u,v)=\frac12\left[\rho_G(u)+\rho_G(v)-\rho_G(v-u)\right]\mbox{ for
all } u, v\in \R.\label{decom-RG}\end{eqnarray} Combining this with
\begin{equation*} Z_{t}=Z_0+e^{-\theta }\int_{0}^{t}e^{-\theta
s}dG_{s},
\end{equation*}
\eqref{cov-wienner-integ-infty} and $\frac{\partial^2R_G}{\partial
v\partial u}(u,v)=\rho_G''(v-u)$ for all $u<v$, we deduce that
\begin{eqnarray*}
E\left(Z_tZ_0\right)&=&e^{-\theta}E\left(Z_0^2\right)+e^{-\theta}\int_{0}^{t}\int_{-\infty}^{0}
e^{\theta u} e^{\theta
v}\frac{\partial^2R_G}{\partial v\partial u}(u,v)dudv\\
&=&e^{-\theta}E\left(Z_0^2\right)+\frac{e^{-\theta t}}{2}
\int_{0}^{t}\int_{-\infty}^{0} e^{\theta u} e^{\theta
v}\rho_G''(v-u)dudv,
\end{eqnarray*}
 which proves \eqref{cov-Z-stat-incr}. The claim \eqref{cov of
X-stat-incr} is a direct consequence of \eqref{cov of X GOU} and
$\frac{\partial^2R_G}{\partial v\partial u}(u,v)=\rho_G''(v-u)$ for
all $u\neq v$. Furthermore, the inequality \eqref{var of
X-stat-incr} follows
immediately from \eqref{link-X-Z} and the stationarity of $Z$.\\
 Now, it remains to prove \eqref{var-Z-stat-incr}. According
to \eqref{IBP-G-infty} and \eqref{decom-RG}, we can write
\begin{eqnarray*}
 E\left(Z_0^2\right)&=&\theta^2 \int_{-\infty}^0 \int_{-\infty}^0e^{\theta u}e^{\theta
 v}R_G(u,v)dudv\\&=&\theta^2  \int_{-\infty}^0\int_{-\infty}^0e^{\theta u}e^{\theta
 v}\rho_G(u)dudv-\theta^2  \int_{-\infty}^0\int_{-\infty}^ve^{\theta u}e^{\theta
 v}\rho_G(v-u)dudv\\
 &=&\theta  \int_{-\infty}^0e^{\theta u}\rho_G(u)du-\theta^2  \int_{-\infty}^0\int_{0}^{\infty}e^{-\theta x}e^{2\theta
 v}\rho_G(x)dxdv\\ &=&\frac{\theta}{2}\int_{0}^{\infty}  e^{-\theta
x}\rho_G(x)dx.
\end{eqnarray*}
Therefore the proof is complete.
\end{proof}

As examples we consider  Hermite Ornstein-Uhlenbeck processes of the
first kind and Gaussian  Ornstein-Uhlenbeck processes of the second
kind and study   the decay of their auto-covariance functions.

\subsection{Fractional  Ornstein-Uhlenbeck processes}

Here we consider the fractional  Ornstein-Uhlenbeck process
\begin{equation}X_t^H := e^{-\theta t} \int_0^te^{\theta s}dB_s^H,\label{XH}\end{equation}
that is, the solution to the Langevin equation \eqref{GOU} in the
case when $G=B^H$ is a fractional Brownian motion with Hurst
parameter
$H\in(0,1)$.\\
Since $B^H$ is Gaussian and
$$E\left(B^H_t-B^H_s\right)^2=|s-t|^{2H};\ s,\ t\geq~0,$$ we deduce
that the assumption $(\mathcal{A})$ holds for $G=B^H$, according to
Remark \ref{wiener-chaos-hypercontractivity}.\\
 So, if $\theta>0$,
the integral \begin{equation} Z_{t}^H:=\int_{-\infty }^{t}e^{-\theta
(t-s)}dB_{s}^H\label{ZH}
\end{equation}
is well defined as a pathwise Riemann-Stieltjes integral and we have
\begin{equation*}
 X_t^H = Z_{t}^H-e^{-\theta t} Z_{0}^H.
\end{equation*}

Let us now state  properties of the processes $X^H$ and $Z^H$,
defined by \eqref{XH} and \eqref{ZH}, respectively.

\begin{theorem}\label{thm-fBm}Assume that $H\in(0,1)$ and $\theta>0$.  Let $X^H$ and
$Z^H$ be the processes defined  by \eqref{XH} and \eqref{ZH},
respectively. Then
\begin{enumerate}
\item[(i)]  $Z^H$ is an ergodic stationary Gaussian process.

\item[(ii)] For any integer $p\geq1$, there exists $C>0$ depending only on
$\theta,H$ and $p$ such that
\begin{itemize}
\item if $p$ is even, $\left|E\left[\left(X^H_t\right)^p\right]-E\left[\left(Z^H_0\right)^p\right]\right|\leq
Ce^{-\theta t}$ for all $t\geq0$, with
\[E\left[\left(Z^H_0\right)^p\right]=\frac{p!}{2^{\frac{p}{2}}\left(\frac{p}{2}\right)!}\left(\frac{H\Gamma(2H)}{\theta^{2H}}\right)^{\frac{p}{2}},\]
\item  if $p$ is odd,
$E\left[\left(X^H_t\right)^p\right]=E\left[\left(Z^H_0\right)^p\right]=0$.
\end{itemize}

\item[(iii)] If   $H\in(0,\frac12)\cup(\frac12,1)$,
$E\left(Z_{t}^HZ_{0}^H\right)\sim\frac{t^{2H-2}}{\theta^2}$ as
$t\rightarrow\infty$. If $H=\frac12$,
$E\left(Z_{t}^{\frac12}Z_{0}^{\frac12}\right)=\frac{e^{-\theta
t}}{2\theta}$.

\item[(iv)] If   $H\in(0,\frac12)\cup(\frac12,1)$,  there exists $C>0$ depending only on
$\theta$ and $H$  such that
$$E\left(X_{t}^HX_{s}^H\right)\leq C|t-s|^{2H-2}\ \mbox{ for all } |t-s|>2,$$ and if
$H=\frac12$, $E\left(X_{t}^{\frac12}X_{s}^{\frac12}\right)\leq
Ce^{-\theta|t-s|}$ for all $|t-s|>2$.
\end{enumerate}
\end{theorem}
\begin{proof}These claims can easily  be obtained
using Theorem \ref{thm-stat-incr} and \eqref{equi1}. The claims
\emph{(i)} and \emph{(iii)} have been obtained previously   by
\cite{CKM} and \cite{BB}.  For the point \emph{(ii)}, it follows
from \cite{EV} that
\[E\left[\left(Z^H_0\right)^p\right]=\left\{
\begin{array}{ll} 0 & \text { if } p \text { is odd}, \\
\frac{p!}{2^{\frac{p}{2}}\left(\frac{p}{2}\right)!}\left(\frac{H\Gamma(2H)}{\theta^{2H}}\right)^{\frac{p}{2}}
& \text { if } p \text { is even}.
\end{array} \right.\]
Combining this  with \emph{(i)}, Gaussianity of $X^H$ and $Z^H$ and
\eqref{link-X-Z}, the claim \emph{(ii)} is obtained.
  The fourth part of Theorem
\ref{thm-fBm} is an obvious consequence of the third part and the
decomposition \eqref{link-X-Z}.

\end{proof}

\begin{remark} When $\theta<0$, the  properties of the processes $X^H$ and $Z^H$
given by \eqref{XH} and \eqref{ZH}, respectively, are very different
from those corresponding to the case $\theta>0$ given in Theorem
\ref{thm-fBm}. For instance, if $\theta<0$, $e^{-\theta
t}X_t^H\longrightarrow\theta\int_0^{\infty}e^{-\theta s}B_s^H ds$
almost surely and in $L^2(\Omega)$ as $t\rightarrow\infty$. We refer
only to \cite{EEO} and \cite{EE} for information about this case and
additional references.
\end{remark}

\subsection{  Hermite  Ornstein-Uhlenbeck processes}

The Hermite process $G^{(q, H)}:=\left\{G_{t}^{(q, H)},
t\in\R\right\}$ of order $q \geq 1$ and  Hurst parameter $H
\in\left(\frac{1}{2}, 1\right)$   is defined as a multiple
Wiener-It\^o integral of the form
\begin{equation}
G_{t}^{(q, H)}=d(q, H) \int_{\mathbb{R}} d W\left(y_{1}\right)
\ldots \int_{\mathbb{R}} d
W\left(y_{q}\right)\left(\int_{0}^{t}\left(s-y_{1}\right)_{+}^{-\left(\frac{1}{2}+\frac{1-H}{q}\right)}
\ldots\left(s-y_{q}\right)_{+}^{-\left(\frac{1}{2}+\frac{1-H}{q}\right)}
d s\right)\label{hermite}
\end{equation} for every $t \in\R$, where $x_{+}^{\alpha}=x^{\alpha} 1_{(0,
\infty)}(x),$ $\int_0^t:=-\int_t^0$ if $t<0$, and $\left\{W(y),y \in
\mathbb{R}\right\}$ is a Wiener process, whereas $d(q, H)$ is a
normalizing positive constant chosen to ensure that
$\mathbf{E}\left[(G_{1}^{(q, H)})^{2}\right]=1$.

Except for Gaussianity, Hermite processes of order $q \geq 2$ share
many properties with the fBm (corresponding to q = 1). First note
that,  according to the fact that $G^{(q, H)}$ is Hermite,
$$E\left(G^{(q, H)}_t-G^{(q, H)}_s\right)^2=|s-t|^{2H};\ s,\ t\in\R,$$ and Remark \ref{wiener-chaos-hypercontractivity}, we deduce
that the assumption $(\mathcal{A})$ holds for $G=G^{(q, H)}$.
Moreover, the Hermite process \eqref{hermite} is $H$-self-similar
and it has stationary increments. Its covariance coincides with the
covariance of the fBm for all $q \geq 1$, that is, for every $q \geq
1$,
\begin{equation}
E\left( G_{t}^{(q, H)} G_{s}^{(q, H)}\right)=\frac{1}{2}\left(t^{2
H}+s^{2 H}-|t-s|^{2 H}\right), \quad s,\ t\in\R.\label{cov-hermite}
\end{equation}
The class of Hermite processes also includes the Rosenblatt process
which is obtained for $q=2$. The Hermite process is non-Gaussian if
$q \geq 2$. These processes have attracted a lot of interest in the
recent past (see the monographs \cite{PT}, \cite{tudor}  and the
references therein).\\
The Wiener integral of a deterministic function $f$ with respect to
a Hermite process $G^{(q, H)}$ which we denote by $\int_{\mathbb{R}}
f(u) d G_{u}^{(q, H)}$, has been constructed by \cite{MT}.\\
We recall that the stochastic integral $\int_{\mathbb{R}} f(u) d
G_{u}^{(q, H)}$ is well-defined for any $f$ belonging to the space
$|\mathcal{H}|$ of functions $f: \mathbb{R} \rightarrow \mathbb{R}$
such that
$$
\int_{\mathbb{R}} \int_{\mathbb{R}}|f(u) f(v)| |u-v|^{2 H-2} d u d
v<\infty.
$$
Further, for any $f, g \in|\mathcal{H}|,$ that \begin{equation}
\mathbb{E}\left[\int_{\mathbb{R}} f(u) d G_{u}^{(q, H)}
\int_{\mathbb{R}} g(v) d G_{u}^{(q, H)}\right]=H(2 H-1)
\int_{\mathbb{R}} \int_{\mathbb{R}} f(u) g(v)|u-v|^{2 H-2} d u d
v.\label{inner-pro-hermite}
\end{equation}
Now, let us consider the Hermite  Ornstein-Uhlenbeck process
\begin{equation}X_t^{(q,H)} := e^{-\theta t} \int_0^te^{\theta s}dG_s^{(q,H)},\quad t\geq0,\label{X-qH}\end{equation}
that is, the solution to the Langevin equation \eqref{GOU} in the
case when $G=G^{(q,H)}$ is the Hermite process  of order $q \geq 1$
and  Hurst parameter $H \in\left(\frac{1}{2}, 1\right)$, according
to Lemma \ref{lemma-Riemann integral}.
\\
Further, if $\theta>0$, the process \begin{equation}
Z_{t}^{(q,H)}:=\int_{-\infty }^{t}e^{-\theta
(t-s)}dG_{s}^{(q,H)},\quad t\geq0,\label{Z-qH}
\end{equation}
is well defined in the Riemann-Stieltjes sense,  and we have
\begin{equation*}
 X_t^{(q,H)} = Z_{t}^{(q,H)}-e^{-\theta t} Z_{0}^{(q,H)},\quad t\geq0.
\end{equation*}
By Theorem \ref{thm-stat-incr}, $Z^{(q,H)}$ is stationary.
Furthermore, using \eqref{inner-pro-hermite}, stationarity of $Z^H$,
and $H>\frac12$, we have, for every $s,\ t\geq0$,
\[E\left(Z_{t}^{(q,H)}Z_{0}^{(q,H)}\right)=E\left(Z_{t}^HZ_{0}^H\right)
\quad \mbox{ for all }\theta>0,\] and
\[E\left(X_{t}^{(q,H)}X_{s}^{(q,H)}\right)=E\left(X_{t}^HX_{s}^H\right)
\quad \mbox{ for all }\theta\in\R, \] where $X^H$ and $Z^H$ are the
processes given by \eqref{XH} and \eqref{ZH}, respectively.
Combining the results above with Theorem \ref{thm-stat-incr} and
Theorem \ref{thm-fBm} leads to the following theorem.
\begin{theorem}\label{thm-hermite}Assume that $H>\frac12$ and $\theta>0$. Let $X^{(q,H)}$ and
$Z^{(q,H)}$ be the processes defined  by \eqref{X-qH} and
\eqref{Z-qH}, respectively. Then
\begin{itemize}
\item  $Z^{(q,H)}$ is a stationary process,
 and
 $E\left[\left(Z^{(q,H)}_0\right)^2\right]=\frac{H\Gamma(2H)}{\theta^{2H}}.$

 \item There exists a constant $C>0$ depending only on $\theta$ and $H$ such that, for all
 $t\geq0$,
  $$\left|E\left[\left(X^{(q,H)}_t\right)^2\right]-\frac{H\Gamma(2H)}{\theta^{2H}}\right|\leq
Ce^{-\theta t}.$$

\item
$E\left(Z_{t}^{(q,H)}Z_{0}^{(q,H)}\right)\sim\frac{t^{2H-2}}{\theta^2}$
as $t\rightarrow\infty$.

\item There exists a constant $C>0$ depending only on $\theta$ and $H$ such that, for all
 $|t-s|>2$,
$E\left(X_{t}^{(q,H)}X_{s}^{(q,H)}\right)\leq C|t-s|^{2H-2}$.
\end{itemize}
\end{theorem}

\subsection{  Gaussian  Ornstein-Uhlenbeck processes of the second
kind} Let  $U:=\{U_t,t\geq0\}$ be a Gaussian process satisfies the
assumption $(\mathcal{A})$. In addition, we assume that the process
$U$ is $\gamma$-self-similar, that is, $\left\{U_{b t}, t \geq
0\right\} \stackrel{Law}{=}\left\{b^{\gamma} U_{t}, t \geq
0\right\}$ for all $b>0$. Hence, the integrals $\int_0^t
e^{-s}dU_{a_s}$ if $t\geq0$ and  $\int_t^0 e^{-s}dU_{a_s}$ if $t<0$,
with $a_t:=\theta e^{\frac{t}{\theta}}$, are well defined as
Riemann-Stieltjes integrals, and let
$Y_{U}^{(1)}:=\left\{Y_{t,U}^{(1)},t\in\R \right\}$ denote the
process defined by  $Y_{t,U}^{(1)}:=\int_0^t e^{-s}dU_{a_s}$ if
$t\geq0$ and $Y_{t,U}^{(1)}:=-\int_t^0 e^{-s}dU_{a_s}$ if $t<0$.
\\
Let us introduce the  following processes,
\begin{equation}\label{eta-def}L_t:=e^{-t}U_{a_t}-U_{a_0},\ t\in\R, \   \eta_t:=\int_{0}^{t}e^{-s}U_{a_s}ds \mbox{ for }
t\geq0,\ \eta_t:=-\int_{t}^{0}e^{-s}U_{a_s}ds \mbox{ for }  t<0.
\end{equation}
Integrating by parts, we get
\begin{equation}\label{Y-decomp}Y_{t,U}^{(1)}=L_t+\eta_t\ \mbox{ for all } t\in\R.
\end{equation}
Define
\begin{eqnarray}f_U(x):=\gamma^{2\gamma}R_U(e^{\frac{x}{2\gamma}},e^{-\frac{x}{2\gamma}})
=\gamma^{2\gamma}E\left(U_{e^{\frac{x}{2\gamma}}}U_{e^{-\frac{x}{2\gamma}}}\right),\quad
x\in\R.\label{f-def}\end{eqnarray}

We will make use of the following lemmas.
\begin{lemma}\label{lemma-Y}
Let $\{L_t,t\in\R\}$ and $\{\eta_t,t\in\R\}$ be the processes given
by (\ref{eta-def}), and let   $f_U$ be the even function defined by
(\ref{f-def}). Then, for every $s,t\in\R$,
\begin{eqnarray}
E\left(\eta_t\eta_s\right)=h_U(s)+h_U(t)-h_U(|t-s|),\label{cov-eta}
\end{eqnarray}
\begin{eqnarray}E\left(L_sL_t\right)=  f_U(|t-s|)-f_U(t)-f_U(s)+f_U(0),\label{cova-L}
\end{eqnarray}
and
\begin{eqnarray}E\left(L_s\eta_t\right)+E\left(L_t\eta_s\right)=0,\label{cova-L-eta}
\end{eqnarray}where $h_U(t):=\int_0^{|t|}(|t|-x)f_U(x)dx$  for all $t\in\R$.\\
 Hence,  for every $s,t\in\R$,
 \begin{eqnarray}R_{Y_{U}^{(1)}}\left(s,t\right)=E\left(Y_{s,U}^{(1)}Y_{t,U}^{(1)}\right)=E\left(L_sL_t\right)+E\left(\eta_s\eta_t\right),\label{cova-Y}
\end{eqnarray}
and, if we suppose that $f_U$ is twice continuously differentiable
on $\R\setminus\{0\}$, we have for every $s,t\in\R$ with $s\neq t$,
 \begin{eqnarray}\frac{\partial^2R_{Y_{U}^{(1)}}}{\partial t\partial
 s}\left(s,t\right)=f_U(|t-s|)-f_U''(|t-s|).
 \label{deriv-cova-Y}
\end{eqnarray}
Moreover, for every $s,t\in\R$,
 \begin{eqnarray}E\left[\left(Y_{t,U}^{(1)}-Y_{s,U}^{(1)}\right)^2\right]=2f_U(0)-2f_U(|t-s|)+2h_U(|t-s|),\label{stat-incr-Y}
\end{eqnarray}
which implies that the Gaussian process $Y_{U}^{(1)}$
has stationary increments.\\
In addition, if we suppose that $\int_{0}^{\infty}
|f_U(x)|dx<\infty$, then the  process $Y_{U}^{(1)}$  satisfies the
assumption $(\mathcal{A})$.
\end{lemma}
\begin{proof}Using similar arguments as in \cite{AAE}, the statements \eqref{cov-eta}-\eqref{cova-Y}
can be immediately proved. For \eqref{deriv-cova-Y}, it follows from
\eqref{cov-eta}-\eqref{cova-Y} that
\[\frac{\partial^2R_{Y_{U}^{(1)}}}{\partial t\partial
 s}\left(s,t\right)=h_U''(|t-s|)-f_U''(|t-s|)\]  for every $s, t\in\R$ such that  $s\neq t$. Furthermore, it is
 clear that $h_U''(x)=f_U(x)$. Thus \eqref{deriv-cova-Y} is
 obtained. The estimate \eqref{stat-incr-Y} follows directly from \eqref{cov-eta}-\eqref{cova-Y}.\\ Let us now prove that $Y_{U}^{(1)}$  satisfies the assumption
$(\mathcal{A})$. Since $U$ is Gaussian, then, using
\eqref{Y-decomp}, the process $Y_{U}^{(1)}$ is Gaussian. Combining
this result with Remark \ref{wiener-chaos-hypercontractivity}, we
see that, in order to show that $Y_{U}^{(1)}$  satisfies
$(\mathcal{A})$, it suffices to prove that for some $C>0$,
\begin{eqnarray*}E\left[\left|Y_{t,U}^{(1)}-Y_{s,U}^{(1)}\right|^{2}\right] \leq
C|t-s|^{\gamma}  \quad \mbox{ for all } s, t \in\R.
\end{eqnarray*}
From \eqref{stat-incr-Y} we have, for every $s,t\in\R$,
 \begin{eqnarray*}E\left[\left(Y_{t,U}^{(1)}-Y_{s,U}^{(1)}\right)^2\right]\leq2|f_U(0)-f_U(|t-s|)|+2|h_U(|t-s|)|.
\end{eqnarray*}
Since $ |h'_U(x)|=|\int_{0}^{x} f_U(x)dx|\leq\int_{0}^{\infty}
|f_U(x)|dx=:C_f<\infty$  for all $x\geq0$, we deduce that
$|h_U(x)|\leq C_f|x|$ for all $x\geq0$.\\
On the other hand, since $U$ is $\gamma$-self-similar, we get for
all $x\geq1$,
\begin{eqnarray*}|f_U(0)-f_U(x)|\leq|f_U(0)|+|f_U(x)|\leq2|f_U(0)|\leq2|f_U(0)||x|^{\gamma}.
\end{eqnarray*}
Moreover, for all $0<x<1$,
\begin{eqnarray*}|f_U(0)-f_U(x)|&=&\gamma^{2\gamma}\left|e^{-x}E\left(U_{1}U_{e^{\frac{x}{\gamma}}}\right)-E\left(U_{1}^2\right)\right|\\
&\leq&\gamma^{2\gamma}\left|(e^{-x}-1)E\left(U_{1}U_{e^{\frac{x}{\gamma}}}\right)\right|+\left|E\left(U_{1}(U_{e^{\frac{x}{\gamma}}}-U_{1})\right)\right|\\
&\leq&C\left(|x|+|x|^{\gamma}\right)\leq C |x|^{\gamma},
\end{eqnarray*} where we used $\frac{e^{-x}-1}{x}\rightarrow-1$ as $x\rightarrow0$
and for all $0<x<1$,
\[\left|E\left(U_{1}(U_{e^{\frac{x}{\gamma}}}-U_{1})\right)\right|\leq
\left(E\left(U_{1}^2\right)\right)^{\frac12}\left(E\left(U_{e^{\frac{x}{\gamma}}}-U_{1}\right)^2\right)^{\frac12}\leq
C\left|\frac{e^{\frac{x}{\gamma}}-1}{x}\right|^{\gamma}|x|^{\gamma}\leq
C|x|^{\gamma},\] according to the assumption $(\mathcal{A})$. Thus
the desired result is obtained.

\end{proof}

\begin{remark}\label{remark-Y1} Since $\{Y_{t,U}^{(1)},t\in\R\}$
has stationary increments, note that
 \begin{eqnarray*}\rho_{Y_{U}^{(1)}}''(t-s)=\frac{\partial^2R_{Y_U^{(1)}}}{\partial t\partial
 s}\left(s,t\right)=f_U(|t-s|)-f_U''(|t-s|),
\end{eqnarray*} according to \eqref{decom-RG} and \eqref{deriv-cova-Y}.
\end{remark}

\begin{lemma}\label{lemma-n-m-gamma}  Define for every  $\gamma>0$,
\[m_{\gamma}(x):=\left(e^{\frac{x}{2\gamma}}-e^{-\frac{x}{2\gamma}}\right)^{2\gamma},\qquad
n_{\gamma}(x):=\left(e^{\frac{x}{2\gamma}}+e^{-\frac{x}{2\gamma}}\right)^{2\gamma}.\]
Then, for every  $\gamma\in(0,\frac12)\cup(\frac12,1)$,
\begin{eqnarray}m_{\gamma}''(x)-m_{\gamma}(x)=\frac{2(2\gamma-1)}{\gamma}\left(e^{\frac{x}{2\gamma}}-e^{-\frac{x}{2\gamma}}\right)^{2\gamma-2},\label{identity-m}
\end{eqnarray}
and
\begin{eqnarray}
n_{\gamma}''(x)-n_{\gamma}(x)=\frac{2(1-2\gamma)}{\gamma}\left(e^{\frac{x}{2\gamma}}+e^{-\frac{x}{2\gamma}}\right)^{2\gamma-2}.\label{identity-n}
\end{eqnarray}
\end{lemma}
\begin{proof}Let us prove \eqref{identity-m}. We have
\begin{eqnarray*}m_{\gamma}''(x)=\left(1-\frac{1}{2\gamma}\right)\left(e^{\frac{x}{2\gamma}}-e^{-\frac{x}{2\gamma}}\right)^{2\gamma-2}
\left(e^{\frac{x}{2\gamma}}+e^{-\frac{x}{2\gamma}}\right)^2+\frac{1}{2\gamma}\left(e^{\frac{x}{2\gamma}}-e^{-\frac{x}{2\gamma}}\right)^{2\gamma}.
\end{eqnarray*}
This leads to
\begin{eqnarray*}m_{\gamma}''(x)-m_{\gamma}(x)&=&\left(1-\frac{1}{2\gamma}\right)\left(e^{\frac{x}{2\gamma}}-e^{-\frac{x}{2\gamma}}\right)^{2\gamma-2}
\left(\left(e^{\frac{x}{2\gamma}}+e^{-\frac{x}{2\gamma}}\right)^2-\left(e^{\frac{x}{2\gamma}}-e^{-\frac{x}{2\gamma}}\right)^{2}\right)\\
&=&4\left(1-\frac{1}{2\gamma}\right)\left(e^{\frac{x}{2\gamma}}-e^{-\frac{x}{2\gamma}}\right)^{2\gamma-2}
\\&=&\frac{2(2\gamma-1)}{\gamma}\left(e^{\frac{x}{2\gamma}}-e^{-\frac{x}{2\gamma}}\right)^{2\gamma-2},
\end{eqnarray*}
which proves \eqref{identity-m}. Similar reasoning gives
\eqref{identity-n}.
\end{proof}

Now, let us consider the Ornstein-Uhlenbeck process of the second
kind $X_{U}:=\left\{X_{t,U},t\geq0\right\}$,
  defined as the unique (pathwise) solution to
\begin{equation}
    X_{0,U}=0,\quad    dX_{t,U}=-\theta X_t dt+dY_{t,U}^{(1)},\quad
t\geq0.  \label{GSK}
\end{equation}
According to Lemma \ref{lemma-Y}, the process $Y_{U}^{(1)}$
satisfies the assumption $(\mathcal{A})$. So, using  Lemma
\ref{lemma-Riemann integral}, the unique solution of (\ref{GSK}) can
be written as
\begin{equation}\label{X-GSK}
X_{t,U}=e^{-\theta t}\int_0^te^{\theta s}dY_{s,U}^{(1)},\quad t\geq
0.
\end{equation}
Moreover,  for any $\theta>0$, the process
\begin{equation}\label{Z-GSK}
Z_{t,U}:=\int_{-\infty }^{t}e^{-\theta (t-s)}dY_{s,U}^{(1)},\quad
t\geq 0
\end{equation}
is  well defined as a Riemann-Stieltjes integral.\\
Hence we can also write
 \begin{equation*}
 X_{t,U} = Z_{t,U}-e^{-\theta t} Z_{0,U},\quad
t\geq 0.
\end{equation*}
Furthermore, since $Y_{U}^{(1)}$ has stationary increments, it
follows from \eqref{var-Z-stat-incr} and \eqref{stat-incr-Y} that
\begin{eqnarray}
E\left(Z_{0,U}^2\right)&=&\theta\int_0^{\infty}\left(f_U(0)-f_U(t)+h_U(t)\right)e^{-\theta
t}dt\nonumber\\
&=&f_U(0)-\theta\int_0^{\infty} f_U(t) e^{-\theta
t}dt+\theta\int_0^{\infty}  e^{-\theta t}\int_0^t(t-x)f_U(x)dx dt
\nonumber\\
&=&f_U(0)-\theta\int_0^{\infty} f_U(t) e^{-\theta
t}dt+\frac{1}{\theta}\int_0^{\infty}  f_U(x)e^{-\theta x} dx\nonumber\\
&=&f_U(0)+\left(\frac{1}{\theta}-\theta\right)\int_0^{\infty}f_U(t)e^{-\theta
t}dt.\label{var of Z-GSK}
\end{eqnarray}

Now we will apply the results above to fractional, subfractional
 and bifractional Ornstein-Uhlenbeck processes of
the second kind.

\subsubsection{Fractional  Ornstein-Uhlenbeck processes of the second
kind}\label{section-FOUSK} Here we consider the fractional
Ornstein-Uhlenbeck process of the second kind
$X_{B^H}:=\left\{X_{t,B^H},t\geq0\right\}$, defined as the unique
solution to \eqref{GSK} when $U=B^H$ is a fBm with Hurst parameter
$H\in(0,1)$. More precisely,
\begin{equation}
    X_{0,B^H}=0,\quad    dX_{t,B^H}=-\theta X_{t,B^H} dt+dY_{t,B^H}^{(1)},\quad
t\geq0.\label{FOUSK}
\end{equation}
 In this case we have, according to \cite[Section 4]{AAE},  $\int_{0}^{\infty}
|f_{B^H}(x)|dx<\infty$. Moreover, for all $x\in\R$,
\begin{eqnarray*}f_{B^H}(x)&=&H^{2H}R_{B^H}(e^{\frac{x}{2 H}},e^{\frac{-x}{2H}})\\ &=&\frac{H^{2H}}{2}\left[e^{x}+e^{-x}
-\left(e^{\frac{x}{2H}}-e^{\frac{-x}{2H}}\right)^{2H}\right]\\
&=&\frac{H^{2H}}{2}\left[e^{x}+e^{-x} - m_H(x)\right],
\end{eqnarray*}
where the function   $m_H(x)$ is defined in Lemma
\ref{lemma-n-m-gamma}. Moreover, using the latter equation, Remark
\ref{remark-Y1} and Lemma \ref{lemma-n-m-gamma}, we get
\begin{eqnarray}
\rho_{Y_{B^H}^{(1)}}''(x)=f_{B^H}(x)-f_{B^H}''(x)= (2H-1)H^{2H-1}
\left(e^{\frac{x}{2H}}-e^{-\frac{x}{2H}}\right)^{2H-2}.\label{identity-rho-fBm}
\end{eqnarray}

\begin{theorem}\label{thm-FOUSK}Assume that  $H\in(0,1)$ and $\theta>0$. Let $\{X_{t,B^H},t\geq0\}$ and $\{Z_{t,B^H},t\geq0\}$ be the processes defined  by
\eqref{X-GSK} and \eqref{Z-GSK} for $U=B^H$, respectively. Then
\begin{itemize}
\item  $\{Z_{t,B^H},t\geq0\}$ is an ergodic stationary Gaussian process,
 and
 \[
E\left(Z_{0,B^H}^2\right)=f_{B^H}(0)+\left(\frac{1}{\theta}-\theta\right)\int_0^{\infty}f_{B^H}(t)e^{-\theta
t}dt.\]

 \item There exists a constant $C>0$ depending only on $\theta$ and $H$ such that, for all
 $t\geq0$,
  $$\left|E\left(X_{t,B^H}^2\right)-E\left(Z_{0,B^H}^2\right)\right|\leq
Ce^{-\theta t}.$$

\item If   $H\in(0,\frac12)\cup(\frac12,1)$, then, as
$t\rightarrow\infty$,
\begin{eqnarray*}E\left(Z_{t,B^H}Z_{0,B^H}\right)&\sim& (2H-1)H^{2H-1}e^{-\min
\left(\theta,\frac{1}{H}-1\right) t}\\&&\quad \times \left\{
\begin{array}{ll} \int_0^{\infty}\left(e^{\theta u}-e^{-\theta
u}\right)\left(e^{\frac{u}{2H}}- e^{-\frac{u}{2H}}\right)^{2H-2}du & \text { if }\ \theta<\frac{1}{H}-1, \\
\frac{t}{2\theta} & \text { if }\ \theta=\frac{1}{H}-1, \\
\frac{1}{\theta^2-\left(\frac{1}{H}-1\right)^2} & \text { if }\
\theta>\frac{1}{H}-1. \end{array} \right.
\end{eqnarray*}

\item  If   $H\in(0,\frac12)\cup(\frac12,1)$, there exists a constant $C>0$ depending only on $\theta$ and $H$ such that, for all
$|t-s|>2$,
\[E\left(X_{t,B^H}X_{s,B^H}\right)\leq C\left\{
\begin{array}{ll}  e^{-\min \left( \theta,\frac{1}{H}-1\right)
|t-s|} & \text { if }\ \theta\neq \frac{1}{H}-1, \\
te^{-\min \left( \theta,\frac{1}{H}-1\right) |t-s|} & \text { if }\
\theta=\frac{1}{H}-1.
\end{array} \right.\]
\end{itemize}
\end{theorem}
\begin{proof}Except the ergodicity of $Z_{B^H}$,  the results of Theorem \ref{thm-FOUSK} can be
immediately obtained by using Theorem \ref{thm-stat-incr},
\eqref{equi2}, \eqref{ineq2}, \eqref{var of Z-GSK} and
\eqref{identity-rho-fBm}. The ergodicity of $Z_{B^H}$   is an
immediate conclusion of the fact that $Z_{B^H}$ is a stationary
Gaussian process  and its auto-covariance function asymptotically
vanishes (see, for instance, \cite[Example 2.2.8]{samorodnitsky}).
\end{proof}

\begin{remark}Note that the third and fourth parts of Theorem \ref{thm-FOUSK} are
valid for all $H\in(0,\frac12)\cup(\frac12,1)$. However, these two
estimates have  been proved in \cite{KS} only when $H\in(\frac12,1)$
with $\theta\neq\frac{1}{H}-1$.
\end{remark}

\subsubsection{Subfractional  Ornstein-Uhlenbeck processes of the second
kind}\label{section-subFOUSK}

The subfractional Brownian motion (subfBm) $S^H:=\{S^H_t,t\geq0\}$
with parameter $H\in(0, 1)$   is a centered Gaussian process with
covariance function
\[E\left(S^H_tS^H_s\right)=t^{2H}+s^{2H}-\frac{1}{2}\left((t+s)^{2H}+|t-s|^{2H}\right).\]
Note that, when $H=\frac12$, $S^{\frac12}$ is a standard Brownian
motion. The subfBm $S^H$ is self-similar with exponent of
self-similarity $H\in(0,1)$ and its increments are non-stationary.
Moreover, using
$$E\left(S^H_t-S^H_s\right)^2\leq (2-2^{2H-1})|s-t|^{2H};\ s,\
t\geq~0,$$ the fact that $S^H$ is Gaussian and Remark
\ref{hypercontractivity}, we deduce that the assumption
$(\mathcal{A})$ holds for $G=S^H$.  On the other hand,  according to
\cite[Section 4]{AAE},  $\int_{0}^{\infty} |f_{S^H}(x)|dx<\infty$.
Moreover, for all $x\in\R$,
\begin{eqnarray*}f_{S^H}(x)&=&R_{S^H}(e^{\frac{x}{2 H}},e^{\frac{-x}{2H}})\\ &=&e^{x}+e^{-x}
-\frac12\left[\left(e^{\frac{x}{2H}}+e^{\frac{-x}{2H}}\right)^{2H}+\left(e^{\frac{x}{2H}}-e^{\frac{-x}{2H}}\right)^{2H}\right]\\
&=&e^{x}+e^{-x} -\frac12\left[n_H(x)+m_H(x)\right],
\end{eqnarray*}
where the functions $n_H(x)$ and $m_H(x)$ are defined in Lemma
\ref{lemma-n-m-gamma}. Further, using the latter equation, Remark
\ref{remark-Y1} and Lemma \ref{lemma-n-m-gamma}, we get
\begin{eqnarray*}
\rho_{Y_{S^H}^{(1)}}''(x)=f_{S^H}(x)-f_{S^H}''(x)=
(2H-1)H^{2H-1}\left[\left(e^{\frac{x}{2H}}+e^{-\frac{x}{2H}}\right)^{2H-2}
-\left(e^{\frac{x}{2H}}-e^{-\frac{x}{2H}}\right)^{2H-2}\right].\label{identity-rho-subfBm}
\end{eqnarray*}

 Now let us consider the subfractional Ornstein-Uhlenbeck process of the second
 kind  $X_{S^H}:=\left\{X_{t,S^H},t\geq0\right\}$,
  defined as the unique  solution to \eqref{GSK} when $U=S^H$.  In other words,
  $X_{S^H}$ is the solution to the equation
\begin{equation*}
    X_{0,S^H}=0,\quad    dX_{t,S^H}=-\theta X_t dt+dY_{t,S^H}^{(1)},\quad
t\geq0,\label{subFOUSK}
\end{equation*}

Using similar arguments as in Section \ref{section-FOUSK}, we deduce
the following result.

\begin{theorem}\label{thm-subFOUSK}Assume that  $H\in(0,1)$ and $\theta>0$.
Let $\{X_{t,S^H},t\geq0\}$ and $\{Z_{t,S^H},t\geq0\}$ be the
processes defined  by \eqref{X-GSK} and \eqref{Z-GSK} for $U=S^H$,
respectively. Then
\begin{itemize}
\item  $\{Z_{t,S^H},t\geq0\}$ is an ergodic stationary Gaussian process,
 and
 \[
E\left(Z_{0,S^H}^2\right)=f_{S^H}(0)+\left(\frac{1}{\theta}-\theta\right)\int_0^{\infty}f_{S^H}(t)e^{-\theta
t}dt.\]

 \item There exists a constant $C>0$ depending only on $\theta$ and $H$ such that, for all
 $t\geq0$,
  $$\left|E\left(X_{t,S^H}^2\right)-E\left(Z_{0,S^H}^2\right)\right|\leq
Ce^{-\theta t}.$$

\item If   $H\in(0,\frac12)\cup(\frac12,1)$, then, as
$t\rightarrow\infty$,
\begin{eqnarray*}&&E\left(Z_{t,S^H}Z_{0,S^H}\right)\sim (2H-1)H^{2H-1} e^{-\min
\left(\theta,\frac{2}{H}-1\right) t}\nonumber
\\&&\quad \times \left\{
\begin{array}{ll} \int_0^{\infty}\left(e^{\theta u}-e^{-\theta
u}\right)\left[\left(e^{\frac{u}{2H}}+
e^{-\frac{u}{2H}}\right)^{2H-2}
-\left(e^{\frac{u}{2H}}- e^{-\frac{u}{2H}}\right)^{2H-2}\right]du & \text { if } \theta<\frac{2}{H}-1, \\
\frac{(2H-2)te^{\frac{t}{H}}}{\theta} & \text { if } \theta=\frac{2}{H}-1, \\
\frac{4H-4}{\theta^2-\left(\frac{2}{H}-1\right)^2} & \text { if }
\theta>\frac{2}{H}-1. \end{array} \right.
\end{eqnarray*}

\item  If   $H\in(0,\frac12)\cup(\frac12,1)$, there exists a constant $C>0$ depending only on $\theta$ and $H$ such that, for all
$|t-s|>2$,
\[E\left(X_{t,S^H}X_{s,S^H}\right)\leq C\left\{
\begin{array}{ll} e^{-\min \left( \theta,\frac{2}{H}-1\right)
|t-s|} & \text { if }\ \theta\neq \frac{2}{H}-1, \\
te^{-\min \left( \theta,\frac{1}{H}-1\right) |t-s|} & \text { if }\
\theta=\frac{2}{H}-1.
\end{array} \right.\]
\end{itemize}
\end{theorem}

\subsubsection{Bifractional  Ornstein-Uhlenbeck processes of the second
kind}\label{section-biFOUSK}

Let $B^{H,K}:=\{B^{H,K}_t,t\geq0\}$ be a bifractional Brownian
motion (bifBm)  with parameters $H\in (0, 1)$ and $K\in(0,1]$. This
means that $B^{H,K}$ is a centered Gaussian process with the
covariance function
\begin{eqnarray*}
E(B^{H,K}_sB^{H,K}_t)=\frac{1}{2^K}\left(\left(t^{2H}+s^{2H}\right)^K-|t-s|^{2HK}\right).
\end{eqnarray*}
  The case $K = 1$ corresponds to the  fBm  with Hurst parameter $H$. The process $B^{H,K}$ verifies \begin{eqnarray*}
 E\left(\left|B^{H,K}_t-B^{H,K}_s\right|^2\right)\leq
2^{1-K}|t-s|^{2HK}.
\end{eqnarray*}
Combining this with the fact that  $B^{H,K}$ is Gaussian and Remark
\ref{hypercontractivity},we deduce that the assumption $(\mathcal{A})$ holds for $G=B^{H,K}$. \\
Furthermore, according to \cite[Section 4]{AAE},  $\int_{0}^{\infty}
|f_{B^{H,K}}(x)|dx<\infty$. We can also write,  for all $x\in\R$,
\begin{eqnarray*}f_{B^{H,K}}(x)&=&R_{B^{H,K}}(e^{\frac{x}{2 HK}},e^{\frac{-x}{2HK}})\\ &=&
\frac{1}{2^K}\left[\left(e^{\frac{x}{K}}+e^{\frac{-x}{K}}\right)^{K}-\left(e^{\frac{x}{2HK}}-e^{\frac{-x}{2HK}}\right)^{2HK}\right]\\
&=&\frac{1}{2^K}\left[n_{\frac{K}{2}}(x)+m_{HK}(x)\right],
\end{eqnarray*}
where the functions $n_{\frac{K}{2}}(x)$ and $m_{HK}$ are defined in
Lemma \ref{lemma-n-m-gamma}. Moreover, using the latter equation,
Remark \ref{remark-Y1} and Lemma \ref{lemma-n-m-gamma}, we get
\begin{eqnarray*}
\rho_{Y_{B^{H,K}}^{(1)}}''(x)&=&f_{B^{H,K}}(x)-f_{B^{H,K}}''(x)\\
&=&\frac{(HK)^{2HK}(K-1)}{2^{K-2}K}\left(e^{\frac{x}{K}}+e^{-\frac{x}{K}}\right)^{K-2}
+\frac{(HK)^{2HK-1}(2HK-1)}{2^{K-1}}\left(e^{\frac{x}{2HK}}-e^{-\frac{x}{2HK}}\right)^{2HK-2}.\label{identity-rho-bifBm}
\end{eqnarray*}

 Now let us consider the bifractional Ornstein-Uhlenbeck process of the second
 kind  $X_{B^{H,K}}:=\left\{X_{t,B^{H,K}},t\geq0\right\}$,
  defined as the unique  solution to \eqref{GSK} when $U=B^{H,K}$.  In other words,
  $X_{B^{H,K}}$ is the solution to the equation
\begin{equation*}
    X_{0,B^{H,K}}=0,\quad    dX_{t,B^{H,K}}=-\theta X_t dt+dY_{t,B^{H,K}}^{(1)},\quad
t\geq0.\label{biFOUSK}
\end{equation*}

Similar arguments as in Section \ref{section-FOUSK} lead the
following result.

\begin{theorem}\label{thm-biFOUSK}Assume that  $H,K\in (0, 1)$ and $\theta>0$.
Let $\{X_{t,B^{H,K}},t\geq0\}$ and $\{Z_{t,B^{H,K}},t\geq0\}$ be the
processes defined  by \eqref{X-GSK} and \eqref{Z-GSK} for
$U=B^{H,K}$, respectively. Then
\begin{itemize}
\item  $\{Z_{t,B^{H,K}},t\geq0\}$ is an ergodic stationary Gaussian process,
 and
 \[
E\left(Z_{0,B^{H,K}}^2\right)=f_{B^{H,K}}(0)+\left(\frac{1}{\theta}-\theta\right)\int_0^{\infty}f_{B^{H,K}}(t)e^{-\theta
t}dt.\]

 \item There exists a constant $C>0$ depending only on $\theta$ and $H$ such that, for all
 $t\geq0$,
  $$\left|E\left(X_{t,B^{H,K}}^2\right)-E\left(Z_{0,B^{H,K}}^2\right)\right|\leq
Ce^{-\theta t}.$$

\item If   $H\in(0,\frac12)$ with $HK\neq\frac12$, then, as
$t\rightarrow\infty$,
\begin{eqnarray*}E\left(Z_{t,B^{H,K}}Z_{0,B^{H,K}}\right)&\sim& \frac{(HK)^{2HK}(K-1)}{2^{K-2}K}e^{-\min
\left(\theta,\frac{2}{K}-1\right) t}\\&&\quad \times \left\{
\begin{array}{ll} \int_0^{\infty}\left(e^{\theta u}-e^{-\theta
u}\right)\left(e^{\frac{u}{K}}+ e^{-\frac{u}{K}}\right)^{K-2}du & \text { if }\ \theta<\frac{2}{K}-1, \\
\frac{t}{2\theta} & \text { if }\ \theta=\frac{2}{K}-1, \\
\frac{1}{\theta^2-\left(\frac{2}{K}-1\right)^2} & \text { if }\
\theta>\frac{2}{K}-1. \end{array} \right.
\end{eqnarray*}

\item If   $H\in(\frac12,1)$ with $HK\neq\frac12$, then, as
$t\rightarrow\infty$,
\begin{eqnarray*}&&E\left(Z_{t,B^{H,K}}Z_{0,B^{H,K}}\right)\sim \frac{(HK)^{2HK-1}(2HK-1)}{2^{K-1}} e^{-\min
\left(\theta,\frac{1}{HK}-1\right) t}\nonumber
\\&&\quad \times \left\{
\begin{array}{ll} \int_0^{\infty}\left(e^{\theta u}-e^{-\theta
u}\right)\left(e^{\frac{u}{2HK}}- e^{-\frac{u}{2HK}}\right)^{2HK-2}du & \text { if }\ \theta<\frac{1}{HK}-1, \\
\frac{t}{2\theta} & \text { if }\ \theta=\frac{1}{HK}-1, \\
\frac{1}{\theta^2-\left(\frac{1}{HK}-1\right)^2} & \text { if }\
\theta>\frac{1}{HK}-1. \end{array} \right.
\end{eqnarray*}

\item  If   $H\in(0,\frac12)$ with $HK\neq\frac12$, there exists a constant $C>0$ depending only on $\theta$ and $H$ such that, for all
$|t-s|>2$,
\[E\left(X_{t,B^{H,K}}X_{s,B^{H,K}}\right)\leq C\left\{
\begin{array}{ll}  e^{-\min
\left(\theta,\frac{2}{K}-1\right)|t-s|} & \text { if }\ \theta\neq \frac{2}{K}-1, \\
te^{-\min \left(\theta,\frac{2}{K}-1\right)|t-s|}& \text { if }\
\theta=\frac{2}{K}-1.
\end{array} \right.\]

\item  If   $H\in(\frac12,1)$ with $HK\neq\frac12$, there exists a constant $C>0$ depending only on $\theta$ and $H$ such that, for all
$|t-s|>2$,
\[E\left(X_{t,B^{H,K}}X_{s,B^{H,K}}\right)\leq C\left\{
\begin{array}{ll} e^{-\min
\left(\theta,\frac{1}{HK}-1\right)|t-s|} & \text { if }\ \theta\neq \frac{1}{HK}-1, \\
te^{-\min \left(\theta,\frac{1}{HK}-1\right) |t-s|} & \text { if }\
\theta=\frac{1}{HK}-1.
\end{array} \right.\]
\end{itemize}
\end{theorem}

\section{Langevin equations driven by  Gaussian processes with non-stationary increments}

This section  deals with  non-stationary Gaussian Ornstein-Uhlenbeck
processes. More precisely, we consider two examples of Gaussian
Ornstein-Uhlenbeck process of the form \eqref{GOU}, where the
driving process $G$ is Gaussian but it does not have stationary
increments.

We will make use of the following technical lemma.
\begin{lemma}\label{key-lemma-nonsta}Let $\gamma\in(0,1)$ and $\lambda>0$.
Then there exists a constant $C>0$ depending only on $\lambda$ and
$\gamma$ such that, for all $t\geq1$,
\begin{eqnarray*}
e^{-\lambda t}\int_0^t s^{\gamma-1}e^{\lambda s}ds\leq C
t^{\gamma-1}.
\end{eqnarray*}
\end{lemma}
\begin{proof}We have, for every $t\geq1$,
\begin{eqnarray*}
e^{-\lambda t}\int_0^t s^{\gamma-1}e^{\lambda s}ds&=&e^{-\lambda
t}\int_0^{\frac{t}{2}} s^{\gamma-1}e^{\lambda s}ds+e^{-\lambda
t}\int_{\frac{t}{2}}^t s^{\gamma-1}e^{\lambda s}ds\\&\leq&e^{-
\frac{\lambda t}{2}}\int_0^{\frac{t}{2}}
s^{\gamma-1}ds+\left(\frac{t}{2}\right)^{\gamma-1}e^{-\lambda
t}\int_{\frac{t}{2}}^t e^{\lambda s}ds\\&\leq&C\left(t^{\gamma}e^{-
\frac{\lambda t}{2}}+t^{\gamma-1}\right)
\\&\leq&Ct^{\gamma-1},
\end{eqnarray*}
which completes the proof.
\end{proof}

\subsection{Subfractional  Ornstein-Uhlenbeck process}
Here  we consider the  Ornstein-Uhlenbeck process
$X^{S^H}:=\{X_{t}^{S^H}, t\geq 0\}$ defined by the following linear
stochastic differential equation
\begin{equation}
X_{0}^{S^H}=0,\qquad dX_{t}^{S^H}=-\theta X_{t}^{S^H}dt+dS^H_{t},
\label{subfOU}
\end{equation} where $S^H$ is a  subfBm with Hurst parameter $H\in(0, 1)$, defined in Section \ref{section-subFOUSK}.
In this case, we can write that, for every $s,t\geq0$,
\begin{eqnarray}R_{S^H}(s,t)&=&R_{B^H}(s,t)+\frac12\left(t^{2H}+s^{2H}-(t+s)^{2H}\right)\nonumber\\
&=:&R_{B^H}(s,t)+g_{S^H}(s,t),\label{decomp-cov-subfBm}
\end{eqnarray} where $B^H$ is a fBm with Hurst parameter $H\in(0,
1)$.\\

Note that the process $S^H$ does not have stationary increments and
so the Gaussian process $Z_{t}^{S^H}:=\int_{-\infty }^{t}e^{-\theta
(t-s)}dS_{s}^H,$ with $\theta>0$, is  non-stationary. Thus, in this
section,  we will only discuss properties of the process $X^{S^H}$.

\begin{theorem}\label{thm-subfOU}Assume that $H\in(0,\frac12)\cup(\frac12,1)$ and $\theta>0$.  Let $X^{S^H}$
 be the process defined  by \eqref{subfOU}. Then, there exists a constant $C>0$ depending only on $\theta$ and $H$ such that, for all
$t>2$,
\begin{eqnarray}
\left|E\left[\left(X^{S^H}_t\right)^2\right]-\frac{H\Gamma(2H)}{\theta^{2H}}\right|\leq
Ct^{2H-2},\label{estimate1-subfOU}
\end{eqnarray}
and for all $|t-s|>2$,
\begin{eqnarray}E\left(X_{t}^{S^H}
X_{s}^{S^H}\right)\leq C|t-s|^{2H-2}.\label{estimate2-subfOU}
\end{eqnarray}
\end{theorem}
\begin{proof}Using \eqref{decomp-var-GOU}, \eqref{decomp-cov-subfBm} and
\eqref{key1}, we deduce that, for every $t>2,$
\begin{eqnarray*}
 E\left[\left(X^{S^H}_t\right)^2\right]&=&
 E\left[\left(X^{H}_t\right)^2\right]+\Delta_{g_{S^H}}(t)
 \\&=&2e^{-2\theta t}\int_0^te^{\theta s} \frac{\partial g_{S^H}
}{\partial s}(s,0)ds+2 e^{-2\theta t}\int_0^tdse^{\theta s}\int_0^s
dr\frac{\partial^2 g_{S^H} }{\partial s\partial r}(s,r)e^{\theta r},
\end{eqnarray*} where $X^H$ is the process given in  Theorem
\ref{thm-fBm}.\\
It is easy to check that for every $s,r>0$,
 \[\frac{\partial g_{S^H}
}{\partial s}(s,0)=0,\ \mbox{ and }\ \frac{\partial^2 g_{S^H}
}{\partial s\partial r}(s,r)=-H(2H-1)\left(r+s\right)^{2H-2}.\]
According to the claim \emph{(ii)} in Theorem \ref{thm-fBm}, we have
\begin{eqnarray*}
\left|E\left[\left(X^{H}_t\right)^2\right]-\frac{H\Gamma(2H)}{\theta^{2H}}\right|\leq
Ce^{-\theta t}, \quad t\geq0.
\end{eqnarray*}
Therefore, in order to prove \eqref{estimate1-subfOU}, it is enough
to check that
 \[e^{-2\theta t}\int_0^tdse^{\theta s}\int_0^s
dr\left(r+s\right)^{2H-2}e^{\theta r}\leq C t^{2H-2}, \quad t>2.\]
On the other hand, for all $t>2$,
\begin{eqnarray*}e^{-2\theta
t}\int_0^tdse^{\theta s}\int_0^s dr\left(r+s\right)^{2H-2}e^{\theta
r}&\leq& e^{-2\theta t}\int_0^tdse^{\theta s}\int_0^s
dr\left(2\sqrt{rs}\right)^{2H-2}e^{\theta r}\\
&=&2^{2H-2}\left(e^{-\theta t}\int_0^t s^{H-1}e^{\theta
s}ds\right)^2\\
&\leq& C t^{2H-2},
\end{eqnarray*}
where the latter equality comes from Lemma \ref{key-lemma-nonsta}.
Thus, \eqref{estimate1-subfOU} is obtained.\\
Now we prove \eqref{estimate2-subfOU}. From Lemma \ref{calculcov},
it follows immediately that, for all $|t-s|>2$,
\begin{eqnarray*}
E\left(X^{S^H}_{s}X^{S^H}_{t}\right) &=& e^{- \theta(t-s)}
E\left[\left(X^{S^H}_{s}\right)^{2}\right] + e^{- \theta t} e^{-
\theta s} \int_{s}^{t} e^{\theta v} \int_{0}^{s} e^{\theta u}
\frac{\partial^2R_{S^H}}{\partial u\partial v} (u,v) du dv\\
&\leq&C|t-s|^{2H-2},
\end{eqnarray*}
where we used \eqref{estimate1-subfOU}, \eqref{ineq1} and the fact
that for every $u,v>0$ with $u\neq v$,
\begin{eqnarray*}\left|\frac{\partial^2R_{S^H}}{\partial u\partial v}
(u,v)\right|&=&H|2H-1|\left|(v+u)^{2H-2}-|v-u|^{2H-2}\right|\\
&\leq&2H|2H-1||v-u|^{2H-2}.
\end{eqnarray*}
\end{proof}

\subsection{Bifractional  Ornstein-Uhlenbeck process}

Consider the  Ornstein-Uhlenbeck process
$X^{B^{H,K}}:=\{X_{t}^{B^{H,K}}, t\geq 0\}$ defined by the following
linear stochastic differential equation
\begin{equation}
X_{0}^{B^{H,K}}=0,\qquad dX_{t}^{B^{H,K}}=-\theta
X_{t}^{B^{H,K}}dt+dB^{H,K}_{t}, \label{bifOU}
\end{equation} where $B^{H,K}$ is a  bifBm with Hurst parameters $H,K\in(0, 1)$, defined in Section \ref{section-biFOUSK}.
In this case, we can write that, for every $s,t\geq0$,
\begin{eqnarray}R_{B^{H,K}}(s,t)&=&\frac{1}{2^{K-1}}R_{B^{HK}}(s,t)+\frac{1}{2^{K}}\left[\left(t^{2H}+s^{2H}\right)^K-t^{2HK}-s^{2HK}\right]\nonumber\\
&=:&\frac{1}{2^{K-1}}R_{B^{HK}}(s,t)+g_{B^{H,K}}(s,t),\label{decomp-cov-bifBm}
\end{eqnarray} where $B^{HK}$ is a fBm with Hurst parameter $HK\in(0,
1)$.\\

Note that the process $S^H$ does not have stationary increments and
so the Gaussian process $Z_{t}^{B^{H,K}}:=\int_{-\infty
}^{t}e^{-\theta (t-s)}dB^{H,K}_s,$ with $\theta>0$, is
non-stationary. Here  we will only discuss properties of the process
$X^{B^{H,K}}$.

\begin{theorem}\label{thm-bifOU}Assume that  $\theta>0$ and  $H,K\in(0, 1)$ with $HK\neq\frac12$.  Let $X^{B^{H,K}}$
 be the process defined  by \eqref{bifOU}. Then, there exists a constant $C>0$ depending only on $\theta$ and $H$ such that, for all
$t>2$,
\begin{eqnarray}
\left|E\left[\left(X^{B^{H,K}}_t\right)^2\right]-\frac{HK\Gamma(2HK)}{2^{K-1}\theta^{2HK}}\right|\leq
Ct^{2HK-2},\label{estimate1-bifOU}
\end{eqnarray}
and for all $|t-s|>2$,
\begin{eqnarray}E\left(X_{t}^{B^{H,K}}
X_{s}^{B^{H,K}}\right)\leq C\left\{\begin{array}{ll}
|t-s|^{2HK-2H-1} & \text { if } 0<H <\frac12 \\
|t-s|^{2HK-2} & \text { if } \frac12\leq
H<1.\end{array}\right..\label{estimate2-bifOU}
\end{eqnarray}
\end{theorem}
\begin{proof}Using \eqref{decomp-var-GOU}, \eqref{decomp-cov-bifBm}  and
\eqref{key1}, we deduce that, for every $t>2,$
\begin{eqnarray*}
 E\left[\left(X^{B^{H,K}}_t\right)^2\right]&=&
 \frac{1}{2^{K-1}}E\left[\left(X^{HK}_t\right)^2\right]+\Delta_{g_{B^{H,K}}}(t)
 \\&=&2e^{-2\theta t}\int_0^te^{\theta s} \frac{\partial g_{B^{H,K}}
}{\partial s}(s,0)ds+2 e^{-2\theta t}\int_0^tdse^{\theta s}\int_0^s
dr\frac{\partial^2 g_{B^{H,K}} }{\partial s\partial r}(s,r)e^{\theta
r},
\end{eqnarray*} where the process $X^{HK}$ is defined in  Theorem
\ref{thm-fBm}.\\
It is easy to check that for every $s,r>0$,
 \[\frac{\partial g_{B^{H,K}}
}{\partial s}(s,0)=0,\ \mbox{ and }\ \frac{\partial^2 g_{B^{H,K}}
}{\partial s\partial
r}(s,r)=\frac{(2H)^2K(K-1)}{2^K}\left(r^{2H}+s^{2H}\right)^{K-2}(rs)^{2H-1}.\]
According to the claim \emph{(ii)} in Theorem \ref{thm-fBm}, we have
\begin{eqnarray*}
\left|E\left[\left(X^{HK}_t\right)^2\right]-\frac{HK\Gamma(2HK)}{\theta^{2HK}}\right|=\mathcal{O}\left(e^{-\theta
t}\right)\ \mbox{ as } t\rightarrow\infty.
\end{eqnarray*}
Thus, in order to prove \eqref{estimate1-bifOU}, it is enough to
check that for every $t>2$,
 \[e^{-2\theta t}\int_0^tdse^{\theta s}\int_0^s
dr\left(r^{2H}+s^{2H}\right)^{K-2}(rs)^{2H-1}e^{\theta r}\leq C
t^{2H-2}.\] On the other hand, for all $t>2$,
\begin{eqnarray*}&&e^{-2\theta
t}\int_0^tdse^{\theta s}\int_0^s
dr\left(r^{2H}+s^{2H}\right)^{K-2}(rs)^{2H-1}e^{\theta r}\\&\leq&
e^{-2\theta t}\int_0^tdse^{\theta s}\int_0^s
dr\left(2r^{H}s^{H}\right)^{K-2}(rs)^{2H-1}e^{\theta r}\\
&=&2^{K-3}\left(e^{-\theta t}\int_0^t s^{HK-1}e^{\theta
s}ds\right)^2\\
&\leq&C t^{2H-2},
\end{eqnarray*}
where the latter equality follows immediately from Lemma
\ref{key-lemma-nonsta}. Therefore, \eqref{estimate1-bifOU} is
obtained.\\
Let us prove \eqref{estimate2-bifOU}. According to Lemma
\ref{calculcov}, we have
\begin{eqnarray}
E\left(X^{B^{H,K}}_{s}X^{B^{H,K}}_{t}\right) &=& e^{- \theta(t-s)}
E\left[\left(X^{B^{H,K}}_{s}\right)^{2}\right] + e^{- \theta t} e^{-
\theta s} \int_{s}^{t} e^{\theta v} \int_{0}^{s} e^{\theta u}
\frac{\partial^2R_{B^{H,K}}}{\partial u\partial v} (u,v) du dv.\nonumber\\
\label{cov-decomp-bifOU}
\end{eqnarray}
If $H<\frac12$, and using Lemma \ref{key-lemma-nonsta}, we have, for
every $s<t$,
\begin{eqnarray*}
 &&e^{- \theta t} e^{-\theta s} \int_{s}^{t} e^{\theta v} \int_{0}^{s} e^{\theta u}
\left|\frac{\partial^2R_{B^{H,K}}}{\partial u\partial v}
(u,v)\right| du dv\\&\leq& Ce^{- \theta t} e^{-\theta s}
\int_{s}^{t} e^{\theta v} \int_{0}^{s} e^{\theta u}
\left(u^{2H}+v^{2H}\right)^{K-2}(uv)^{2H-1}dudv\\
&\leq& Ce^{- \theta t} e^{-\theta s} \int_{s}^{t} e^{\theta v}
\int_{0}^{s} e^{\theta u} \left(v^{2H}\right)^{K-2}(uv)^{2H-1}dudv\\
&=& C\left[e^{-\theta s}
\int_{0}^{s} e^{\theta u} u^{2H-1}du\right]\left[e^{- \theta t}  \int_{s}^{t} e^{\theta v} v^{2HK-2H-1}dv\right]\\
&\leq& Ce^{- \theta t}  \int_{s}^{t} e^{\theta v} v^{2HK-2H-1}dv\\
&=& Ce^{- \theta (t-s)}  \int_{0}^{t-s} e^{\theta y}
(s+y)^{2HK-2H-1}dy\\
&=& C\left[e^{- \theta (t-s)}  \int_{0}^{\frac{t-s}{2}} e^{\theta y}
(s+y)^{2HK-2H-1}dy+e^{- \theta (t-s)}  \int_{\frac{t-s}{2}}^{t-s}
e^{\theta y} (s+y)^{2HK-2H-1}dy\right]
\\&\leq& C\left[e^{- \theta \frac{(t-s)}{2}}
+(t-s)^{2HK-2H-1}\right].
\end{eqnarray*}
Combining this with \eqref{cov-decomp-bifOU} and \eqref{estimate1-bifOU}, we obtain \eqref{estimate2-bifOU} for $H<\frac12$.\\
 If
$H\geq\frac12$, then \eqref{estimate2-bifOU} holds, due to
\eqref{cov-decomp-bifOU}, \eqref{estimate1-bifOU}, \eqref{ineq1} and
the fact that, for every $0<u<v$,
\begin{eqnarray*}\left|\frac{\partial^2R_{B^{H,K}}}{\partial u\partial v}
(u,v)\right|&=&\frac{(2H)^2K(1-K)}{2^K}\left(u^{2H}+v^{2H}\right)^{K-2}(uv)^{2H-1}\\
&\leq&\frac{(2H)^2K(1-K)}{2^K}\left(2u^{H}v^{H}\right)^{K-2}(uv)^{2H-1}\\
&\leq&\frac{(2H)^2K(1-K)}{4}v^{2HK-2}\\
&\leq&\frac{(2H)^2K(1-K)}{4}|v-u|^{2HK-2}.
\end{eqnarray*}
\end{proof}

\noindent \textbf{Acknowledgments}\\

\noindent  I thank  the two anonymous reviewers for
their helpful comments and suggestions.\\


\begin{thebibliography}{99}
\bibitem{AAE}  Alazemi, F., Alsenafi, A.,   Es-Sebaiy, K. (2020). Parameter
estimation for Gaussian mean-reverting Ornstein-Uhlenbeck processes
of the second kind: non-ergodic case.  Stochastics and Dynamics
20(2), 2050011 (25 pages).

\bibitem{BET} Balde, M. F.,  Es-Sebaiy, K.  and  Tudor, C. A.
(2020). Ergodicity and drift parameter estimation for
infinite-dimensional fractional Ornstein-Uhlenbeck process of the
second kind.  Applied Mathematics and Optimization,  81, 785-814.

\bibitem{BB} Barndorff-Nielsen O.E., Basse-O'Connor A. (2011). Quasi
Ornstein-Uhlenbeck processes. Bernoulli 17, 916-941



\bibitem{CKM} Cheridito, P., Kawaguchi, H., Maejima, M. (2003). Fractional
Ornstein-Uhlenbeck processes, Electr. J. Prob. 8, 1-14.


\bibitem{CV12a} Chronopoulou, A.,   Viens, F. (2012) Estimation and pricing under
long-memory stochastic volatility. Annals of Finance 8, 379-403.


\bibitem{CV12b} Chronopoulou, A.,   Viens, F. (2012) Stochastic volatility and
option pricing with long-memory in discrete and continuous time.
Quantitative Finance 12, 635-649.

 \bibitem{CCR} Comte, F.,  Coutin, L.,   Renault, E.
(2012) Affine fractional stochastic volatility models. Annals of
Finance 8, 337-378.

\bibitem{CR98} Comte, F., Renault, E. (1998). Long memory in continuous-time stochastic volatility
models. Mathematical Finance, 8(4):291-323.

\bibitem{DEKN} Douissi, S., Es-Sebaiy, K., Kerchev, G., Nourdin I. (2022). Berry-esseen bounds of second moment estimators for gaussian
processes observed at high frequency, Electron. J. Statist. 16(1):
636-670. DOI: 10.1214/21-EJS1967

\bibitem{DEV} Douissi, S., Es-Sebaiy, K., Viens, F. (2019). Berry-Esseen
bounds for parameter estimation of general Gaussian processes.
\emph{ALEA, Lat. Am. J. Probab. Math. Stat.}, \textbf{16}, 633-664.

\bibitem{EEO} El Machkouri, M., Es-Sebaiy, K.,  Ouknine, Y. (2016). Least squares
estimator for non-ergodic Ornstein-Uhlenbeck processes driven by
Gaussian processes. Journal of the Korean Statistical Society 45,
329-341.



\bibitem{EE}  Es-Sebaiy, K. and   Es.Sebaiy, M. (2021). Estimating drift parameters in a non-ergodic Gaussian Vasicek-type model.
 Statistical Methods and Applications 30, 409-436.


\bibitem{EV}   Es-Sebaiy, K., Viens, F.  (2019). Optimal rates for parameter estimation of
stationary Gaussian processes. Stochastic Processes and their
Applications, 129(9), 3018-3054.


\bibitem{GJR}  Gatheral, J.,  Jaisson, T., Rosenbaum,  M. (2018). Volatility Is Rough.
Quantitative Finance, 18(6), 933-949.

\bibitem{HN}  Hu, Y.,  Nualart, D. (2010). Parameter estimation for fractional
    Ornstein-Uhlenbeck processes.  Statist. Probab. Lett. {\bf 80},
    1030-1038.



\bibitem{KS}  Kaarakka, T., Salminen, P. (2011). On Fractional
Ornstein-Uhlenbeck process. Communications on Stochastic Analysis,
5, 121-133.

\bibitem{KM} K\v{r}\'{a}\v{z}, P., Maslowski, B.  (2019). Central limit theorems and
minimum-contrast estimators for linear stochastic evolution
equations. Stochastics, 91, 1109-1140.

\bibitem{lindgren} Lindgren, G. (2012). Stationary stochastic processes: theory and
applications.  Chapman and Hall/CRC.


\bibitem{MT}  Maejima M., Tudor C.A. (2007). Wiener Integrals with respect to the
Hermite process and a non-central limit theorem. Stoch. Anal. Appl.
25(5), 1043-1056.

\bibitem{NP-book} Nourdin, I., Peccati, G. (2012).
 Normal approximations with Malliavin calculus : from Stein's method
to universality. Cambridge Tracts in Mathematics 192. Cambridge
University Press, Cambridge.


\bibitem{PT} Pipiras, V., Taqqu, M. (2017). Long-range dependence and self-similarity.
Cambridge series in statistical and probabilistic mathematics.
Cambridge University Press, Cambridge.

\bibitem{samorodnitsky} Samorodnitsky G. (2016). Stochastic processes and long range
dependence. Springer series in operations research and financial
engineering. Springer, Cham.


\bibitem{shao} Shao, Y. (1995). The fractional Ornstein-Uhlenbeck process as a
representation of homogeneous Eulerian velocity turbulence. Phys. D
83 461-477. MR1334532.

\bibitem{SV} Sottinen, T., Viitasaari, L. (2018). Parameter estimation for the
Langevin equation with stationary-increment Gaussian noise.
Statistical Inference for Stochastic Processes, 21(3), 569-601.

\bibitem{tudor} Tudor, C.A. (2013). Analysis of variations for self-similar processes: A stochastic
calculus approach. Probability and its Applications. Springer.

\bibitem{WZ}Wheeden, R.L. and Zygmund, A., (2015). Measure and Integral: An Introduction to Real
Analysis, Second Edition. New York, Chapman and Hall/CRC.

\bibitem{Young} Young, L.C. (1936). An inequality of the H\"older type connected with Stieltjes
integration,  Acta Mathematica, 67(1),  251-282.


\end{thebibliography}
\end{document}